\documentclass[twocolumn]{autart}    

\usepackage{cite}
\usepackage{amsmath,amssymb,amsfonts}
\usepackage{graphicx}
\usepackage{algorithm,algorithmic}
\usepackage{hyperref}
\usepackage{textcomp}

\usepackage{xcolor}

    \usepackage{amsmath} 
    \usepackage{amssymb}  
    \usepackage{stmaryrd} 
    \usepackage{amssymb}  
    \usepackage{empheq} 

    \usepackage{graphicx} 
    \usepackage[hang,small,bf]{caption}
    \usepackage[subrefformat=parens]{subcaption}
    \usepackage{stfloats}
    \fnbelowfloat

    \newcommand{\vertiii}[1]{{\left\vert\kern-0.25ex\left\vert\kern-0.25ex\left\vert #1 
        \right\vert\kern-0.25ex\right\vert\kern-0.25ex\right\vert}}

    \newenvironment{proof}[1][Proof]{%
      \par\noindent\textbf{#1. }\ignorespaces
    }{%
      \hfill$\square$\par
    }

\begin{document}

\begin{frontmatter}

\title{Mutual Information Optimal Density Control of Linear Systems and Generalized Schr\"{o}dinger Bridges \\with Reference Refinement} 

\thanks[footnoteinfo]{This paper was not presented at any IFAC 
meeting. Corresponding author K. Kashima. 
}

\author[Kyoto]{Shoju Enami}\ead{enami.shoujyu.57r@st.kyoto-u.ac.jp},    
\author[Kyoto]{Kenji Kashima}\ead{kk@i.kyoto-u.ac.jp},               

\address[Kyoto]{Graduate School of Informatics, Kyoto University, Kyoto, Japan}  


\begin{abstract}                          
We consider a mutual information (MI) regularized optimal density control of a discrete-time linear system.
MI optimal control has been utilized for exploration in reinforcement learning and privacy protection in control.
MI regularization induces stochasticity in the policy, which poses challenges for applications of MI optimal control in safety-critical scenarios.
To remedy this situation, we impose Gaussian density constraints at specified times to directly control state uncertainty.
For this MI optimal density control problem, we propose an alternating optimization algorithm and investigate its convergence properties.
In addition, we reveal a relationship between the MI optimal density control problem and a so-called generalized Schr\"{o}dinger bridge problem associated with the discrete-time linear system.
Based on the results of MI optimal density control and this relationship, we also investigate alternating optimization for the Schr\"{o}dinger bridge problem and its convergence properties.

\end{abstract}

\end{frontmatter}

\section{Introduction}\label{sec:Introduction}

Density control provides a framework for regulating stochastic state uncertainty by steering the state distribution toward a prescribed target distribution.
It has been studied under various settings, including linear systems \cite{bakolas2016optimal, bakolas2018finite, liu2024optimal}, nonlinear systems \cite{yi2020nonlinear, caluya2021nonlinear, elamvazhuthi2018nonlinear}, chance constraints \cite{okamoto2018chance}, and maximum entropy (MaxEnt) optimal control \cite{ito2023maximum, ito2024maximum}.

Mutual information (MI) regularization has been studied in reinforcement learning (RL) \cite{grau2018soft, leibfried2020mutual, malloy2020deep} and optimal control \cite{cundy2024privacy, enami2025mutual, enami2025policy} for applications such as exploration and privacy protection.
Since policies under MI regularization are stochastic in general and therefore induce uncertainty in the state, combining MI optimal control with density control is a natural direction.
However, to the best of our knowledge, MI optimal density control has not yet been investigated.

The Schr\"{o}dinger bridge, which has received renewed attention in data sciences \cite{peyre2019computational}, serves as an alternative framework for steering the state distribution of a dynamical system toward a specified target.
The Schr\"{o}dinger bridge problem \cite{schrodinger1931uber, schrodinger1932sur} seeks the most likely stochastic process (referred to as \textit{controlled process}) that matches two prescribed marginal distributions at two different time instants while remaining as close as possible to a given stochastic process (referred to as \textit{reference process}).
Various extensions of Schr\"{o}dinger bridges have been proposed.
In the generalized Schr\"{o}dinger bridge, we can incorporate desired specifications of the state of the controlled process by adding regularization terms \cite{chen2023GSB, liu2023GSB, theodoropoulos2026GSB}.
The Schr\"{o}dinger bridge with reference refinement is another extension, which alternately optimizes the controlled and reference processes \cite{morimoto2025linear, shen2024multi}.
In particular, \cite{shen2024multi} established its convergence under convexity assumptions on the family of reference process distributions based on the classical alternating minimization theory for KL divergence.

Optimal density control and the Schr\"{o}dinger bridge are deeply interconnected.
Recent literature \cite{chen2016relation, ito2023maximum} has established equivalence between optimal density control and Schr\"{o}dinger bridge problems, demonstrating that the latter can be systematically reduced to the former and solved leveraging established tools from control theory.
The corresponding Schr\"{o}dinger bridge formulation for MI optimal density control remains unelucidated.

\paragraph*{Contributions}
The main contributions of this paper are listed as follows:
\begin{enumerate}
    \item We investigate an MI optimal density control problem for discrete-time linear systems with initial and terminal densities are both Gaussian.
    Specifically, we propose an alternating optimization algorithm of this problem, which optimizes the policy and the prior alternately.
    In addition, we analyze convergence properties of this algorithm such as existence of a cluster point, admissibility of such a cluster point, and coordinate-wise optimality of it;
    
    \item We establish a correspondence between the MI optimal density control problem and a generalized Schr\"{o}dinger bridge problem with reference refinement.
    Exploiting this correspondence, we show that the alternating optimization steps for the Schr\"{o}dinger bridge problem can be obtained directly from the policy and prior updates of the MI optimal density control problem.
    Consequently, the convergence results for the MI optimal density control algorithm yield corresponding convergence properties for alternating optimization of the Schr\"{o}dinger bridge problem.
    Notably, unlike existing convergence results based on convex families of reference process distributions, the reference process class arising from our formulation is generally nonconvex.
\end{enumerate}


\paragraph*{Organization}
This paper is organized as follows: In Section \ref{sec:Formulation of MI Optimal Density Control}, we formulate an MI optimal density control problem.
Section \ref{sec:Alternating Optimization for MI Optimal Density Control} proposes an alternating optimization algorithm for the MI optimal density control problem and analyzes its convergence properties.
In Section \ref{sec:Formulation of Generalized Schrodinger Bridges}, we formulate a generalized Schr\"{o}dinger bridge problem with reference refinement associated with the linear system.
Section \ref{sec:Alternating Optimization for Schrodinger Bridges} starts by revealing a relationship between these two problems and then proposes an alternating optimization algorithm for the Schr\"{o}dinger bridge problem with convergence analysis.
Section \ref{sec:Numerical Experiment} provides some numerical examples.
We conclude this paper in Section \ref{sec:Conclusion}.

\paragraph*{Notation}
The set of all integers that are larger than or equal to $a$ is denoted by $\mathbb{Z}_{\geq a}$.
The set of integers $\{k,k+1,\ldots, l\}(k<l)$ is denoted by $\llbracket k, l\rrbracket$.
The set of all symmetric matrices of size $n$ is denoted by $\mathbb{S}^{n}$.
For $A, B \in \mathbb{S}^{n}$, we write $A \succ B$ (resp. $A \succeq B$) if $A-B$ is positive definite (resp. positive semidefinite).
The identity matrix is denoted by $I$, and its dimension depends on the context.
The Euclidean norm and the Frobenius norm are denoted by $\|\cdot \|$.
The determinant and the trace of $A \in \mathbb{R}^{n\times n}$ are denoted by $|A|$ and $\mathrm{Tr}(A)$, respectively.
For $A \in \mathbb{R}^{n\times m}$, denote the image of $A$ by $\mathrm{Im}(A)$.
For $x \in \mathbb{R}^{n}$ and $A \in \mathbb{S}^{n}$, denote $\|x\|_{A} := (x^{\top}Ax)^{\frac{1}{2}}$.
Note that $\|\cdot\|_{A}$ is no longer a norm if $A$ is not positive definite.
For $A \in \mathbb{R}^{n \times m}$, denote the Moore-Penrose inverse of $A$ by $A^{\dagger}$.
The expected value of a random variable under probability law $p$ is denoted by $\mathbb{E}_{p}[\ \cdot\ ]$, and the subscript $p$ is omitted if the context is clear.
A multivariate Gaussian distribution with mean $\mu \in \mathbb{R}^{n}$ and covariance matrix $\Sigma \succeq 0$ is denoted by $\mathcal{N}(\mu,\Sigma)$.
For a probability distribution $p$, we also use $p$ to denote its probability density function (PDF) when it exists.
Similarly, we denote the PDF of $\mathcal{N}(\mu,\Sigma)$ by the same symbol when it exists.
Denote by $\mathrm{Leb}^{m}$ the Lebesgue measure on $\mathbb{R}^{m}$.
For probability distributions $p$ and $q$, the Radon–Nikodym derivative and the KL divergence between $p$ and $q$ are denoted by $\frac{dp}{dq}$ and $\mathcal{D}_{ \mathrm{KL}}[p \| q]$, respectively, when they are defined.
For probability distributions $\{p^{(i)}\}_{i\in \mathbb{Z}_{\geq 0}}$ and $p$, $p^{(i)}\rightarrow p$ denotes weak convergence.
We use the same symbol for a random variable and its realization.
Throughout this paper, a point is called admissible for an optimization problem if it satisfies all constraints and has a finite objective value.
In addition, an optimization problem is called admissible if it admits at least one admissible point.

\section{Formulation of MI Optimal Density Control} \label{sec:Formulation of MI Optimal Density Control}

Now, we formulate an MI optimal density control problem of linear systems as follows.

\begin{prob}
    Find a pair of a policy $\pi = \{\pi_{k}\}_{k=0}^{T-1}$ and a prior $\rho = \{\rho_{k}\}_{k=0}^{T-1}$ that solves
    \begin{align}
        &\inf_{\pi, \rho \in \mathcal{R}} J(\pi,\rho):=\mathbb{E}\left[ \sum_{k=0}^{T-1} \frac{1}{2}\|u_{k}\|^{2} + \mathcal{D}_{ \mathrm{KL}}[\pi_{k}(\cdot|x_{k}) \| \rho_{k}]  \right] \label{eq:objective function of MIODCP}\\
        &\mbox{s.t. }x_{k+1} = A_{k} x_{k} + B_{k} u_{k} + w_{k}, \label{eq:linear system of MIODCP}\\
        &\hspace{19pt}u_{k} \sim \pi_{k}(\cdot|x) \ \mbox{given }x=x_{k}, \label{eq:stochastic input of MIODCP}\\
        &\hspace{19pt}x_{0} \sim \mathcal{N}(0, \Sigma_{ \mathrm{ini}}), x_{T} \sim \mathcal{N}(0, \Sigma_{ \mathrm{fin}}), \label{eq:boundary conditions of MIODCP}
    \end{align}
    where $T \in \mathbb{Z}_{\geq 1}, x_{k} \in \mathbb{R}^{n}, u_{k} \in \mathbb{R}^{m}, A_{k} \in \mathbb{R}^{n \times n}, B_{k} \in \mathbb{R}^{n \times m}, \Sigma_{ \mathrm{ini}}, \Sigma_{ \mathrm{fin}} \succ 0$.
    The system noises $w_{0},\ldots,w_{T-1}$ are mutually independent, where $w_{k}\sim\mathcal N(0,W_{k})$, and $w_{k}$ is independent of $x_{0},u_{0},\ldots,u_{k}$.
    A feedback policy $\pi_{k}$ denotes a conditional probability distribution on $\mathbb{R}^{m}$ given $x=x_{k}$.
    A prior $\rho_{k}$ denotes a probability distribution on $\mathbb{R}^{m}$. 
    The prior class $\mathcal{R}$ is defined as
        \begin{align}
        \mathcal{R} := \left\{  \rho = \{ \rho_{k} \}_{k=0}^{T-1} \mid \rho_{k} = \mathcal{N}(0, R_{k}),R_{k} \succeq 0 \right\}. \label{eq:prior class}
    \end{align}
    \label{prob:MIODCP}
\end{prob}

It is known that minimizing the KL divergence in \eqref{eq:objective function of MIODCP} over all probability distributions $\rho$ on $\mathbb{R}^{m}$ gives
\begin{align*}
    \min_{\rho}\sum_{k=0}^{T-1}\mathbb{E}[\mathcal{D}_{\mathrm{KL}}[\pi_{k}(\cdot|x_{k})\|\rho_{k}]] = \sum_{k=0}^{T-1}\mathcal{I}(x_{k};u_{k}),
\end{align*}
where $\mathcal{I}(x_{k};u_{k})$ is the mutual information between $x_{k}$ and $u_{k}$.
This is why we call Problem \ref{prob:MIODCP} the MIODCP.

Formally fixing the prior as $\rho_{k}=\mathrm{Leb}^{m}$ yields $\mathcal{D}_{\mathrm{KL}}[\pi_{k}(\cdot|x)\|\rho_{k}]=-\mathcal{H}(\pi_{k}(\cdot|x)) $, where $\mathcal{H}(p)$ is the entropy of $p$.
This implies that mutual information regularization is a mathematical extension of entropy regularization.
Entropy regularization has been utilized in RL for exploration.
Mutual information regularization is used to adjust the exploration effect by joint prior optimization in RL \cite{grau2018soft, leibfried2020mutual, malloy2020deep}.
Mutual information regularization has been also used in model-based control for the sake of privacy protection \cite{cundy2024privacy, enami2025policy}.
In settings where an eavesdropper knows the system and the implemented controller and observes the current input $u_{k}$, but not the state $x_{k}$, $\mathcal{I}(x_{k};u_{k})$ functions as a criterion of how much information of $x_{k}$ is conveyed through $u_{k}$.

\begin{rem}
    One may consider a more general stage cost $\frac{1}{2}\|u_{k}\|_{\mathbf{R}_{k}}^{2} + \varepsilon \mathcal{D}_{\mathrm{KL}}[\pi_{k}(\cdot|x_{k})\|\rho_{k}]$ with $\mathbf{R}_{k}\succ 0$ and $\varepsilon>0$.
    Indeed, from the variable change $\hat{u}_{k}:=\mathbf{R}_{k}^{1/2}u_{k}/\sqrt{\varepsilon}$, the MI optimal density control problem with a general stage cost can be reduced to Problem \ref{prob:MIODCP} with input matrix $\sqrt{\varepsilon}B_{k}\mathbf{R}_{k}^{-1/2}$.
\end{rem}

\begin{rem}
    Problem \ref{prob:MIODCP} with nonzero mean boundary constraint $x_{0}\sim \mathcal{N}(\mu_{ \mathrm{ini}},\Sigma_{ \mathrm{ini}}), x_{T}\sim \mathcal{N}(\mu_{ \mathrm{fin}},\Sigma_{ \mathrm{fin}})$ is also of interest.
    In fact, this problem can be decomposed into an optimal control problem of the state mean and Problem \ref{prob:MIODCP} with zero mean boundary constraint \eqref{eq:boundary conditions of MIODCP}.
    For the details, see Appendix \ref{app:MI Optimal Density Control with Nonzero Initial and Terminal Means}.
    \label{rem:MIODCP with nonzero mean marginal constraints}
\end{rem}

\section{Alternating Optimization for MI Optimal Density Control} \label{sec:Alternating Optimization for MI Optimal Density Control}

Although we consider tractable linear-quadratic settings with Gaussian distributions, the calculation of an optimal solution to Problem \ref{prob:MIODCP} is still challenging due to its nonconvexity.
Instead of an optimal solution, inspired by \cite{leibfried2020mutual, enami2025mutual}, this section considers alternating optimization of the policy and the prior.
We first consider the optimal policy and the optimal prior for a fixed prior and a fixed policy, respectively.
Combining them, we construct the alternating optimization algorithm and analyze its convergence properties.

\subsection{Optimal Policy for Fixed Prior}
\label{subsec:Optimal Policy for Fixed Prior}

In Section \ref{subsec:Optimal Policy for Fixed Prior}, we only consider nondegenerate priors belonging to
\begin{align*}
    \mathcal{R}_{+}:=\left\{  \rho = \{ \rho_{k} \}_{k=0}^{T-1} \mid \rho_{k} = \mathcal{N}(0, R_{k}),R_{k} \succ 0 \right\}.
\end{align*}
The reason why we consider optimal policies for fixed nondegenerate priors will be discussed in Section \ref{subsec:Alternating Optimization Algorithm for MIODCP}.
Note that we can extend the following results with nondegenerate prior class $\mathcal{R}_{+}$ to degenerate prior class $\mathcal{R}$ as referred to in Remark \ref{rem:extension to degenerate prior class}.

Since Problem \ref{prob:MIODCP} with fixed prior $\rho \in \mathcal{R}_{+}$ is an infinite dimensional optimization problem, calculating an optimal policy for fixed $\rho$ is not tractable in general.
Against this issue, the rest of Section \ref{subsec:Optimal Policy for Fixed Prior} will show that Problem \ref{prob:MIODCP} with fixed $\rho \in \mathcal{R}_{+}$ can be actually reduced to a finite dimensional convex optimization problem.

As preparation, we first show that the state covariance $X_{k}:=\mathbb{E}[(x_{k}-m_{k})(x_{k}-m_{k})^{\top}]$, where $m_{k}:=\mathbb{E}[x_{k}]$ is the state mean, is positive definite under the invertibility of $A_{k}$.
\begin{lem}
    Suppose that $A_{k}$ is invertible for any $k\in\llbracket0,T-1\rrbracket$.
    For a fixed prior $\rho\in \mathcal{R}_{+}$, let $\pi$ be any admissible policy of Problem \ref{prob:MIODCP} with fixed $\rho$.
    Then, $X_{k}\succ 0$ for any $k \in \llbracket 0,T \rrbracket$.
    \label{lem:positive state covariance}
\end{lem}

\begin{proof}
In this proof, denote the probability distribution of $x_{k}$ under $\pi$ by $p_{x_{k}}$.
In addition, denote the variance of a random number $a\in \mathbb{R}$ by $\mathrm{Var}(a)$.
Since $\pi$ is admissible, its objective value is finite.
In particular, for any $k\in\llbracket 0,T-1\rrbracket$, $\mathbb{E}_{p_{x_{k}}}[\mathcal{D}_{\mathrm{KL}}[\pi_{k}(\cdot|x_{k})\|\rho_{k}]]<\infty$.
It hence follows that $\mathcal{D}_{\mathrm{KL}}[\pi_{k}(\cdot|x)\|\rho_{k}]<\infty$ for $p_{x_{k}}$-a.e. $x$, which implies $\pi_{k}(\cdot|x)\ll \rho_{k}\ll \mathrm{Leb}^{m}$ for $p_{x_{k}}$-a.e. $x$.
With $X_{0}=\Sigma_{\mathrm{ini}}\succ 0$, we will prove $X_{k}\succ 0$ by induction.
Suppose that $X_{k}\succ0$ for some $k\in\llbracket0,T-1\rrbracket$.
To the contrary, assume that $X_{k+1}$ is not invertible.
Then, there exists a nonzero vector $z$ such that $z^{\top}X_{k+1}z=0$, that is, the variance of $z^{\top}x_{k+1}=z^{\top}(A_{k}x_{k}+B_{k}u_{k}+w_{k})$ is zero.
Using the independence of $w_{k}$ from $(x_{k},u_{k})$, we obtain
\begin{align*}
    \mathrm{Var}(z^{\top}(A_{k}x_{k}+B_{k}u_{k}))+z^{\top}W_{k}z=0.
\end{align*}
Since both terms are nonnegative, $\mathrm{Var}(z^{\top}(A_{k}x_{k}+B_{k}u_{k}))=0$.
Therefore, there exists a constant $c\in\mathbb{R}$ such that $z^{\top}(A_{k}x_{k}+B_{k}u_{k})=c$ a.s.
Consequently, for $p_{x_{k}}$-a.e. $x$, $z^{\top}B_{k}u_{k}=c-z^{\top}A_{k}x_{k}$ $\pi_{k}(\cdot|x)$-a.s.
If $B_{k}^{\top} z\neq0$, the set $\{u\in \mathbb{R}^{m}\mid z^{\top}B_{k}u_{k}=c-z^{\top}A_{k}x_{k}\}$ is an affine hyperplane of Lebesgue measure zero.
This contradicts $\pi_{k}(\cdot| x)\ll\mathrm{Leb}^{m}$.
Hence, $B_{k}^{\top}z=0$ and thus $z^{\top} A_{k}x_{k}=c$ a.s., which implies
\begin{align*}
    \mathrm{Var}(z^{\top}A_{k}x_{k})=z^{\top}A_{k}X_{k}A_{k}^{\top}z=0
\end{align*}
Since $A_{k}X_{k}A_{k}^{\top}\succ0$ from the assumption $X_{k}\succ 0$ and the invertibility of $A_{k}$, we have $z=0$, which contradicts the choice of $z$.
Therefore, $X_{k+1}\succ0$.
The claim of Lemma \ref{lem:positive state covariance} follows by induction.
\end{proof}

Note that the invertibility of $A_{k}$ is supposed to guarantee $X_{k}\succ 0$.
This assumption has been adopted in previous studies of covariance steering \cite{ito2023maximum, liu2024optimal}.
In many situations, $A_{k}$ is expected to be invertible for discrete-time systems.
For example, obtaining the system \eqref{eq:linear system of MIODCP} by a zero-order hold discretization of a continuous-time linear system implies the invertibility of $A_{k}$.
We will maintain this assumption henceforth.
Using Lemma \ref{lem:positive state covariance}, we show that an optimal policy is a linear Gaussian policy.

\begin{prop}
    Consider a given prior $\rho \in \mathcal{R}_{+},\rho_{k}=\mathcal{N}(0,R_{k})$.
    Suppose that $A_{k}$ is invertible for any $k \in \llbracket 0,T-1 \rrbracket$.
    Then, for any admissible policy $\pi$ of Problem \ref{prob:MIODCP} with fixed $\rho$, there exists an admissible linear Gaussian policy
    \begin{align*}
        &\tilde{\pi}_{k}(\cdot|x)=\mathcal{N}(K_{k}x,V_{k}),\\
        &K_{k}\in \mathbb{R}^{m\times n}, V_{k}\succ 0\ \forall k \in \llbracket0,T-1\rrbracket
    \end{align*}
    such that $J(\tilde{\pi},\rho)\leq J(\pi,\rho)$.
    In addition, the equality  $J(\tilde{\pi},\rho)= J(\pi,\rho)$ holds only if $\pi_{k}(\cdot|x)=\tilde{\pi}_{k}(\cdot|x)$ for $\mathrm{Leb}^{n}$-a.e. $x$ and any $k\in\llbracket0,T-1\rrbracket$.
    \label{prop:optimal policy is linear Gaussian}
\end{prop}

\begin{proof}
    We proceed in two steps: successive reductions to a zero mean policy $\check{\pi}$ and a linear Gaussian policy $\tilde{\pi}$.
    In this proof, denote by $p_{k},p_{x_{k}}$, and $p_{u_{k}}$ the probability distributions of $(x_{k},u_{k}),x_{k}$, and $u_{k}$ under policy $\pi$, respectively.
    Similarly, for policies $\check{\pi}$ and $\tilde{\pi}$, denote $\check{p}_{k},\check{p}_{x_{k}}, \check{p}_{u_{k}}, \tilde{p}_{k},\tilde{p}_{x_{k}}$, and $\tilde{p}_{u_{k}}$.

    In this proof, let $m_{k}:=\mathbb{E}_{p_{x_{k}}}[x_{k}], \check{x}_{k}:=x_{k}-m_{k}, a_{k}:=\mathbb{E}_{p_{u_{k}}}[u_{k}], \check{u}_{k}:=u_{k}-a_{k}$.
    Then, the time evolutions of the mean $m_{k}$ and the deviation $\check{x}_{k}$ under policy $\pi$ are given by
    \begin{align*}
        &m_{k+1}= A_{k}m_{k}+B_{k}a_{k},m_{0}=m_{T}=0,\\
        &\check{x}_{k+1}=\ A_{k}\check{x}_{k}+B_{k}\check{u}_{k}+w_{k},\\
        &\check{x}_{0}\sim \mathcal{N}(0,\Sigma_{\mathrm{ini}}), \check{x}_{T}\sim \mathcal{N}(0,\Sigma_{\mathrm{fin}}).
    \end{align*}
    Define translated policy $\check\pi_k(u|x):=\pi_k(u+a_{k}|x+m_{k})$.
    Under $\check{\pi}$, the corresponding state and input have the same laws as $\check{x}_{k}$ and $\check{u}_{k}$, respectively, and thus $\check{\pi}$ is also feasible.
    In addition, $\mathbb{E}_{\check{p}_{x_{k}}}[x_{k}]=0$ and $\mathbb{E}_{\check{p}_{u_{k}}}[u_{k}]=0$.
    Since $\check{p}(x_{k})=p(x_{k}+m_{k})$ by construction, it follows that
    \begin{align*}
        &\mathbb{E}_{\check{p}_{x_{k}}}[\mathcal{D}_{\mathrm{KL}}[\check{\pi}_{k}(\cdot|x_{k})\| \rho_{k}]] \\
        &=\mathbb{E}_{p_{x_{k}}}[\mathcal{D}_{\mathrm{KL}}[\pi_{k}(\cdot|x_{k})\| \rho_{k}]] -\frac{1}{2}\|a_{k}\|_{R_{k}^{-1}}^{2},
    \end{align*}
    while
    \begin{align*}
        \mathbb{E}_{\check{p}_{u_{k}}}[\|u_{k}\|^{2}] = \mathbb{E}_{p_{u_{k}}}[\|u_{k}\|^{2}] - \|a_{k}\|^{2}.
    \end{align*}
    Therefore,
    \begin{align*}
        J(\check{\pi},\rho)= J(\pi,\rho) -\frac{1}{2}\sum_{k=0}^{T-1}\|a_{k}\|_{I+R_{k}^{-1}}^{2} \leq J(\pi,\rho)
    \end{align*}
    and $\check{\pi}$ is admissible.

    Here, let us abuse the notation $X_{k}$ as $X_{k}:=\mathbb{E}_{\check{p}_{k}}[x_{k}x_{k}^{\top}]$.
    From Lemma \ref{lem:positive state covariance}, $X_{k}\succ 0$ for any $k \in \llbracket0,T\rrbracket$.
    In addition, under $\check{\pi}$, define $Y_{k}:=\mathbb{E}_{\check{p}_{k}}[u_{k}x_{k}^{\top}]$ and $U_{k}:=\mathbb{E}_{\check{p}_{u_{k}}}[u_{k}u_{k}^{\top}]$.
    Using them, we also define $K_{k}:=Y_{k}X_{k}^{-1}$ and $V_k:=U_{k}-Y_{k}X_{k}^{-1}Y_{k}^{\top}$.
    Since
    \begin{align*}
        V_k=\mathbb{E}_{\check{p}_{k}}[(u_{k}-K_{k}x_{k})(u_{k}-K_{k}x_{k})^{\top}],
    \end{align*}
    we have $V_{k}\succeq 0$.
    If $z^{\top} V_{k}z=0$ for some $z\neq0$, then
    \begin{align*}
        z^{\top} u_{k}=z^{\top} K_{k}x_{k}\ \text{a.s.},
    \end{align*}
    which again forces $\check{\pi}_{k}(\cdot|x_{k})$ to be concentrated on an affine hyperplane of zero Lebesgue measure for $\check{p}_{x_{k}}$-a.e. $x_{k}$, which contradicts absolute continuity.
    It hence follows that $V_{k}\succ0$.
    Now define $\tilde{\pi}_{k}(\cdot|x):=\mathcal N(K_{k}x,V_{k})$.
    By construction,
    \begin{align*}
        &\mathbb{E}_{\tilde{p}_{x_{k}}}[x_{k}]=0,\mathbb{E}_{\tilde{p}_{u_{k}}}[u_{k}]=0,\\
        &\mathbb{E}_{\tilde{p}_{x_{k}}}[x_{k}x_{k}^{\top}]=X_{k},\mathbb{E}_{\tilde{p}_{k}}[u_{k}x_{k}^{\top}]=Y_{k}, \mathbb{E}_{\tilde{p}_{u_{k}}}[u_{k}u_{k}^{\top}]=U_{k}.
    \end{align*}
    Hence, the covariance recursion under $\tilde{\pi}$ coincides with that under $\check{\pi}$, and induction gives $x_{k}\sim \mathcal{N}(0,X_{k})$ under $\tilde{\pi}$.
    In particular, $\tilde{\pi}$ satisfies the terminal Gaussian constraint \eqref{eq:boundary conditions of MIODCP} and is therefore feasible.
    Policies $\tilde{\pi}$ and $\check{\pi}$ have the same input cost
    \begin{align*}
        \mathbb{E}_{\tilde{p}_{u_{k}}}[\|u_{k}\|^{2}] = \mathrm{Tr}[U_{k}] = \mathbb{E}_{\check{p}_{u_{k}}}[\|u_{k}\|^{2}].
    \end{align*}
    In addition, since $\log\frac{\tilde{\pi}_{k}(u|x)}{\rho_{k}(u)}$ is a quadratic function of $(x,u)$ and the probability laws $\tilde{p}_{k}$ and $\check{p}_{k}$ have the same first and second moments, it follows that
    \begin{align*}
        \mathbb{E}_{\check{p}_{x_{k}}}\left[\mathcal{D}_{\mathrm{KL}}[\check{\pi}_{k}(\cdot|x_{k})\|\rho_{k}]\right]=&\mathbb{E}_{\check{p}_{x_{k}}}\left[\mathcal{D}_{\mathrm{KL}}[\check{\pi}_{k}(\cdot|x_{k})\|\tilde{\pi}_{k}(\cdot|x_{k})]\right]\\
        &+ \mathbb{E}_{\tilde{p}_{x_{k}}}\left[\mathcal{D}_{\mathrm{KL}}[\tilde{\pi}_{k}(\cdot|x_{k})\|\rho_{k}]\right].
    \end{align*}
    Therefore,
    \begin{align}
        J&(\check{\pi},\rho)-J(\tilde{\pi},\rho)\nonumber\\
        &=\sum_{k=0}^{T-1}\mathbb{E}_{\check{p}_{x_{k}}}\left[\mathcal{D}_{\mathrm{KL}}[\check{\pi}_{k}(\cdot|x_{k})\|\tilde{\pi}_{k}(\cdot|x_{k})]\right]\geq 0. \nonumber
    \end{align}
    Finally, we obtain
    \begin{align*}
        J(\tilde{\pi},\rho) \leq J(\check{\pi},\rho) \leq J(\pi,\rho).
    \end{align*}
    Suppose that equality $J(\tilde{\pi},\rho)=J(\pi,\rho)$ holds.
    Then, both inequalities above hold with equality.
    Since $I+R_{k}^{-1}\succ0$, the equality $J(\check{\pi},\rho)=J(\pi,\rho)$ implies $a_{k}=0$ for any $k\in\llbracket0,T-1\rrbracket$.
    Together with $m_{0}=0$ and $m_{k+1}=A_{k}m_{k}+B_{k}a_{k}$, this yields $m_{k}=0$ for any $k\in\llbracket0,T\rrbracket$.
    Hence, $\check{\pi}=\pi$.
    In addition, the equality $J(\check{\pi},\rho)=J(\widetilde{\pi},\rho)$ implies $\mathcal{D}_{\mathrm{KL}}[\check{\pi}_{k}(\cdot|x)\|\tilde{\pi}_{k}(\cdot|x)]=0$ for $\check{p}_{x_{k}}$-a.e. $x$ and any $k\in\llbracket0,T-1\rrbracket$.
    Therefore, $\check{\pi}_{k}(\cdot|x)=\tilde{\pi}_{k}(\cdot|x)$ for $\check{p}_{x_{k}}$-a.e. $x$.
    Since $\check{\pi}$ and $\tilde{\pi}$ have the same initial state distribution, the above equality implies inductively that the state distributions under $\check{\pi}$ and $\tilde{\pi}$ coincide at any $k$.
    Consequently, $\check{p}_{x_{k}} = \tilde{p}_{x_{k}}=\mathcal{N}(0,X_{k})$.
    By Lemma \ref{lem:positive state covariance}, $X_{k}\succ0$, and hence $\check{p}_{x_{k}}$ and $\mathrm{Leb}^{n}$ are mutually absolutely continuous.
    Since $\check{\pi}=\pi$, we conclude that $\pi_{k}(\cdot|x)=\tilde{\pi}_{k}(\cdot|x)$ for $\mathrm{Leb}^{n}$-a.e. $x$ and any $k\in\llbracket0,T-1\rrbracket$.
    This completes the proof.
\end{proof}

Based on Proposition \ref{prop:optimal policy is linear Gaussian}, we only consider the following linear Gaussian policy class
\begin{align}
    \mathcal{P}_{+}:=\{&\pi=\{\pi_{k}\}_{k=0}^{T-1}\mid \pi_{k}(\cdot|x)=\mathcal{N}(K_{k}x,V_{k}),\nonumber \\
    &K_{k}\in \mathbb{R}^{m\times n}, V_{k}\succ 0\} \label{eq:policy class}
\end{align}
without loss of generality.
Using this restriction of the policy class, we reduce Problem \ref{prob:MIODCP} with fixed $\rho \in \mathcal{R}_{+}$ to a finite dimensional convex problem.
To this end, let us abuse the notations
\begin{align*}
    X_{k}:=\mathbb{E}[x_{k}x_{k}^{\top}], Y_{k}:=K_{k}X_{k}, U_{k}:=K_{k}X_{k}K_{k}^{\top}+V_{k}.
\end{align*}
Under policy $\pi \in \mathcal{P}_{+}$, the covariance dynamics can be rewritten as
\begin{align}
    X_{k+1}=&(A_{k}+B_{k}K_{k})X_{k}(A_{k}+B_{k}K_{k})^{\top}\nonumber\\
    &+B_{k}V_{k}B_{k}^{\top}+W_{k}\nonumber\\
    =&A_{k}X_{k}A_{k}^{\top} + A_{k}Y_{k}^{\top}B_{k} + B_{k}Y_{k}A_{k}^{\top}\nonumber\\
    &+B_{k}U_{k}B_{k}^{\top}+W_{k}, \label{eq:covariance dynamics in convex problem}
\end{align}
which is affine in $(X_{k},Y_{k},U_{k})$.
In addition, by denoting $M_{k}:=I+R_{k}^{-1}$ and noting that $X_{k}\succ 0$ for any $k \in \llbracket0,T\rrbracket$ with the invertibility of $A_{k}$ from Lemma \ref{lem:positive state covariance}, the stage cost can be rewritten as
\begin{align*}
    &\mathbb{E}\left[\frac{1}{2}\|u_{k}\|^{2} + \mathcal{D}_{\mathrm{KL}}[\pi_{k}(\cdot|x_{k})\|\rho_{k}]\right]\\
    &=\frac{1}{2}\left(\mathrm{Tr}[M_{k}U_{k}]-\log|U_{k}-Y_{k}X_{k}^{-1}Y_{k}^{\top}| \right) + \mathrm{Const}.
\end{align*}
To obtain a standard linear matrix inequality (LMI) representation of  $-\log|U_{k}-Y_{k}X_{k}^{-1}Y_{k}^{\top}|$, we introduce $S_{k}\succ0$ with $S_{k} \preceq U_{k}-Y_{k}X_{k}^{-1}Y_{k}^{\top}$, which is equivalently expressed by
\begin{align}
    \begin{bmatrix}
        X_{k} & Y_{k}^{\top}\\ Y_{k} & U_{k}-S_{k}
    \end{bmatrix}
    \succeq 0 \label{eq:LMI in convex problem}
\end{align}
from the Schur complement \cite{boyd2004convex}.
Because $-\log|S_{k}|$ is strictly decreasing in the Loewner order, this inequality is tight, that is, $S_{k}=U_{k}-Y_{k}X_{k}^{-1}Y_{k}^{\top}$ at optimality.
Therefore, Problem \ref{prob:MIODCP} with fixed $\rho \in \mathcal{R}_{+}$ can be reduced to the following convex problem.

\begin{prob}
    \begin{align}
        &\inf_{\{X_{k}\}_{k=0}^{T},\{Y_{k},U_{k},S_{k}\}_{k=0}^{T-1}} \sum_{k=0}^{T-1}\mathrm{Tr}[M_{k}U_{k}]-\log|S_{k}| \nonumber\\
        &\mathrm{s.t.}\ X_{0}=\Sigma_{\mathrm{ini}}, X_{T}=\Sigma_{\mathrm{fin}},\nonumber \\
        &\hspace{17pt}\eqref{eq:covariance dynamics in convex problem}, \eqref{eq:LMI in convex problem}, S_{k}\succ 0\ \forall k \in \llbracket 0,T-1 \rrbracket .\nonumber
    \end{align}
    \label{prob:convex problem}
\end{prob}

\begin{rem}
    We give a remark on the admissibility of Problem \ref{prob:convex problem}.
    For Problem \ref{prob:convex problem}, admissibility and feasibility are equivalent.
    This feasibility is independent of the fixed prior $\rho\in \mathcal{R}_{+}$ since $R_{k}\succ 0$ in the prior appears only in the objective.
    It hence amounts to the covariance reachability of \eqref{eq:covariance dynamics in convex problem} from $\Sigma_{\mathrm{ini}}$ to $\Sigma_{\mathrm{fin}}$ under a linear Gaussian policy.
    A simple sufficient condition for the reachability can be obtained from, for example, \cite{liu2024optimal}.
    Specifically, by denoting the state transition matrix
    \begin{align*}
        \Phi(k,l):=\begin{cases}
            A_{k-1}A_{k-2}\cdots A_{l}, &l< k,\\
            I, &l= k,
        \end{cases}
    \end{align*}
    and the reachability Gramian
    \begin{align*}
        G(k,l):=\sum_{i=l}^{k-1}\Phi(k,i+1)B_{i}B_{i}^{\top}\Phi(k,i+1)^{\top},l<k,
    \end{align*}
    if $\Sigma_{\mathrm{fin}}\succ W_{T-1}$, $G(T,0)\succ 0$, $A_{k}$ is invertible for any $k\in \llbracket0,T-1 \rrbracket$, and
    \begin{align*}
        \mathrm{Im}(\Phi(T,k)[W_{k-1}^{1/2}\ B_{k-1}])\subset \mathrm{Im}(G(T,k))
    \end{align*}
    for any $k\in \llbracket1,T-1 \rrbracket$, then the reachability is guaranteed, which also implies the feasibility of Problem \ref{prob:convex problem}.
    Note that, for a fixed prior $\rho\in\mathcal{R}_{+}$, under the invertibility of $A_{k}$, Problem \ref{prob:MIODCP} is admissible if and only if Problem \ref{prob:convex problem} is admissible, where the forward implication follows from Proposition \ref{prop:optimal policy is linear Gaussian} and the backward implication is trivial.
    For more details of the covariance reachability of discrete-time linear systems, see \cite{liu2026reachability}.
    \label{rem:admissibility of convex problem}
\end{rem}

Even under admissibility, it is not trivial whether Problem \ref{prob:convex problem} has a unique optimal solution or not.
The following proposition answers this question in the affirmative.

\begin{prop}
    Consider a given prior $\rho\in \mathcal{R}_{+},\rho_{k}=\mathcal{N}(0,R_{k})$.
    Suppose that $A_{k}$ is invertible for any $k \in \llbracket0,T-1 \rrbracket$ and Problem \ref{prob:MIODCP} with fixed $\rho$ has an admissible solution.
    Then, Problem \ref{prob:convex problem} has a unique optimal solution $\{X_{k}^{\rho}\}_{k=0}^{T},\{Y_{k}^{\rho},U_{k}^{\rho},S_{k}^{\rho}\}_{k=0}^{T-1}$.
    \label{prop:existence and uniqueness of convex problem}
\end{prop}

\begin{proof}
    We first show the existence of an optimal solution.
    Note that the assumptions of the admissibility of Problem \ref{prob:MIODCP} and the invertibility of $A_{k}$ imply the admissibility of Problem \ref{prob:convex problem}.
    From this admissibility, there exists a constant such that the sublevel set $\mathcal{C}$ consisting of all feasible solutions with objective values less than or equal to the constant is not empty.
    Recalling that $X_{k}\succ 0$ under the invertibility assumption, the LMI \eqref{eq:LMI in convex problem} implies $U_{k}-S_{k}-Y_{k}X_{k}^{-1}Y_{k}^{\top}\succeq 0$.
    By denoting $\alpha_{k}:=\lambda_{\min}(M_{k})>0$, the stage cost can be evaluated as
    \begin{align*}
        g_{k}(U_{k},S_{k}):=&\mathrm{Tr}[M_{k}U_{k}]-\log |S_{k}|\\
        \geq &\alpha_{k}\mathrm{Tr}[S_{k}]-\log |S_{k}|.
    \end{align*}
    The right-hand side is coercive on $S_{k}\succ 0$, that is, it tends to $\infty$ as $\lambda_{\min}(S_{k})\downarrow 0$ or $\|S_{k}\|\rightarrow \infty$.
    Hence, $\mathcal{C}$ satisfies $\underline{s}_{k}I\preceq S_{k} \preceq \bar{s}_{k}I$ with some $\underline{s}_{k}>0$ and $\bar{s}_{k}>0$ for any $k \in \llbracket0,T-1 \rrbracket$.
    In addition, $U_{k}\succeq S_{k}$ and boundedness of $\mathrm{Tr}[M_{k}U_{k}]$ over $\mathcal{C}$ with $M_{k}\succ 0$ imply boundedness of $U_{k}$ in $\mathcal{C}$.
    Starting from $X_{0}=\Sigma_{\mathrm{ini}}$, the LMI \eqref{eq:LMI in convex problem} implies that $Y_{k}$ is bounded if $X_{k}$ and $U_{k}-S_{k}$ are bounded.
    The covariance recursion \eqref{eq:covariance dynamics in convex problem} then bounds $X_{k+1}$.
    By induction, all variables $\{X_{k}\}_{k=0}^{T},\{Y_{k},U_{k},S_{k}\}_{k=0}^{T-1}$ in $\mathcal{C}$ are bounded.
    Since $S_{k}$ is bounded away from singularity, the constraint $S_{k}\succ 0$ can be replaced with a closed condition $S_{k}\succeq \underline{s}_{k}I$.
    In addition, the other constraints are closed.
    Hence, $\mathcal{C}$ is a bounded, closed, and finite dimensional set, and thus it is compact.
    Combining compactness of $\mathcal{C}$ with continuity of the objective on $\mathcal{C}$ therefore implies the existence of an optimal solution.

    Next, we show uniqueness of an optimal solution.
    At any optimal solution, the LMI \eqref{eq:LMI in convex problem} is tight and the Schur complement yields
    \begin{align*}
        S_{k}=U_{k} - Y_{k}X_{k}^{-1}Y_{k}^{\top}.
    \end{align*}
    By eliminating $S_{k}$ and abusing the notation $g_{k}$, the stage cost can be rewritten as
    \begin{align*}
        g_{k}(X_{k},Y_{k},U_{k})=\mathrm{Tr}[M_{k}U_{k}]-\log|U_{k} - Y_{k}X_{k}^{-1}Y_{k}^{\top}|.
    \end{align*}
    For any two points $(X_{k}^{[1]},Y_{k}^{[1]},U_{k}^{[1]})$ and $(X_{k}^{[2]},Y_{k}^{[2]},U_{k}^{[2]})$ in the domain and any $\theta\in[0,1]$, we have
    \begin{align}
        &U_{k}^{\theta}-Y_{k}^{\theta}(X_{k}^{\theta})^{-1}(Y_{k}^{\theta})^{\top}\nonumber\\
        &\succeq \theta\left(U_{k}^{[1]}-Y_{k}^{[1]}(X_{k}^{[1]})^{-1}(Y_{k}^{[1]})^{\top}\right)\nonumber\\
        &\quad+(1-\theta)\left(U_{k}^{[2]}-Y_{k}^{[2]}(X_{k}^{[2]})^{-1}(Y_{k}^{[2]})^{\top}\right),
    \end{align}
    where $(X_{k}^{\theta},Y_{k}^{\theta},U_{k}^{\theta})$ denotes the corresponding convex combination.
    Since $-\log|\cdot|$ is convex and nonincreasing with respect to the Loewner order, and $\mathrm{Tr}[M_{k}U_{k}]$ is affine, $g_{k}$ is jointly convex in $(X_{k},Y_{k},U_{k})$.
    In addition, for fixed $X_{k}\succ0$, the Schur complement gives
    \begin{align*}
        &g_{k}(X_{k},Y_{k},U_{k})\\
        &=\mathrm{Tr}[M_{k}U_{k}] -\log\left|
        \begin{bmatrix}
            X_{k} & Y_{k}^{\top}\\
            Y_{k} & U_{k}
        \end{bmatrix}
        \right| +\log| X_{k}|.
    \end{align*}
    Since $-\log|\cdot|$ is strictly convex and the mapping
    \begin{align}
        (Y_{k},U_{k}) \mapsto
        \begin{bmatrix}
            X_{k} & Y_{k}^{\top}\\
            Y_{k} & U_{k}
        \end{bmatrix} \succ 0
    \end{align}
    is affine and injective, $g_{k}$ is strictly convex in $(Y_{k},U_{k})$ for fixed $X_{k}$.
    Suppose now that there exist two distinct optimal solutions $\{X_{k}^{[1]}\}_{k=0}^{T},\{Y_{k}^{[1]},U_{k}^{[1]}\}_{k=0}^{T-1}$ and $\{X_{k}^{[2]}\}_{k=0}^{T}$, $\{Y_{k}^{[2]},U_{k}^{[2]}\}_{k=0}^{T-1}$.
    Let $j$ be the first time step at which $(Y_{j}^{[1]},U_{j}^{[1]})\neq (Y_{j}^{[2]},U_{j}^{[2]})$.
    Since $X_{0}^{[1]}=X_{0}^{[2]}=\Sigma_{\mathrm{ini}}$, equality of all preceding $(Y_{k},U_{k})$ implies $X_{j}^{[1]}=X_{j}^{[2]}$.
    Since the constraints of Problem \ref{prob:convex problem} are convex, the midpoint of the two solutions is feasible.
    By the strict convexity of $g_{j}$ in $(Y_{j},U_{j})$,
    \begin{align*}
        &g_{j}\left(X_{j},\frac{Y_{j}^{[1]}+Y_{j}^{[2]}}{2}, \frac{U_{j}^{[1]}+U_{j}^{[2]}}{2}\right)\\
        &< \frac{1}{2}\left\{g_{j}\left(X_{j},Y_{j}^{[1]}, U_{j}^{[1]}\right)  +  g_{j}\left(X_{j},Y_{j}^{[2]}, U_{j}^{[2]}\right)\right\},
    \end{align*}
    whereas joint convexity gives the corresponding nonstrict inequality at all other time stages.
    Hence the midpoint has a strictly smaller total cost than the two optimal solutions, which is a contradiction.
    Therefore, the optimal solution $\{X_{k}^{\rho}\}_{k=0}^{T},\{Y_{k}^{\rho},U_{k}^{\rho}\}_{k=0}^{T-1}$ is unique.
    The corresponding $S_{k}^{\rho}=U_{k}^{\rho} - Y_{k}^{\rho}(X_{k}^{\rho})^{-1}(Y_{k}^{\rho})^{\top}$ is uniquely determined, which completes the proof.
\end{proof}

The unique optimal solution to Problem \ref{prob:convex problem} can be efficiently calculated by existing solvers of convex optimization.
Using this optimal solution, define the corresponding optimal policy of Problem \ref{prob:MIODCP} with fixed $\rho \in \mathcal{R}_{+}$ as
\begin{align}
    \pi^{\rho} \in \mathcal{P}_{+},\pi_{k}^{\rho}(\cdot|x)=\mathcal{N}(K_{k}^{\rho}x,V_{k}^{\rho}), \label{eq:optimal policy for fixed prior}
\end{align}
where $K_{k}^{\rho}:=Y_{k}^{\rho}(X_{k}^{\rho})^{-1}, V_{k}^{\rho}:=U_{k}^{\rho}-Y_{k}^{\rho}(X_{k}^{\rho})^{-1}(Y_{k}^{\rho})^{\top}$.
The following corollary from Propositions \ref{prop:optimal policy is linear Gaussian} and \ref{prop:existence and uniqueness of convex problem} shows that $\pi^{\rho}$ is also the unique optimal policy of Problem \ref{prob:MIODCP} with fixed $\rho \in \mathcal{R}_{+}$.

\begin{cor}
    Consider a given prior $\rho \in \mathcal{R}_{+}$.
    Suppose the assumptions of Proposition \ref{prop:existence and uniqueness of convex problem}.
    Then, $\pi^{\rho}$ is the unique optimal policy of Problem \ref{prob:MIODCP} with fixed $\rho$ in the sense that for any optimal policy $\pi^{*}$ for the same fixed prior $\rho$ and any $k\in\llbracket0,T-1\rrbracket$, $\pi_{k}^{*}(\cdot|x)=\pi_{k}^{\rho}(\cdot|x)$ for $\mathrm{Leb}^{n}$-a.e. $x$. 
    \label{cor:uniqueness of optimal policy for fixed prior}
\end{cor}

\begin{rem}
    With the assumption that $W_{k}=0$ for any $k \in \llbracket0,T-1 \rrbracket$ and some technical assumptions, the closed form of $\pi^{\rho}$ can be derived by reducing Problem \ref{prob:MIODCP} with fixed $\rho\in \mathcal{R}_{+}$ to a MaxEnt optimal density control problem.
    In fact, the stage cost can be rewritten as
    \begin{align*}
        &\mathbb{E}\left[\frac{1}{2}\|u_{k}\|^{2} + \mathcal{D}_{\mathrm{KL}}[\pi_{k}(\cdot|x_{k})\|\rho_{k}]\right]\\
        &=\mathbb{E}\left[\frac{1}{2}\|u_{k}\|_{I+R_{k}^{-1}}^{2} - \mathcal{H}(\pi_{k}(\cdot|x_{k}))\right] + \mathrm{Const.}
    \end{align*}
    Then, we can apply the results of \cite{ito2023maximum}, which shows the closed form of the optimal policy of the MaxEnt optimal density control problem.
\end{rem}

\begin{rem}
    The policy update in Section \ref{subsec:Optimal Policy for Fixed Prior} can be extended to a fixed degenerate prior $\rho\in\mathcal{R}$.
    Suppose $R_{k}\neq 0$ and let $E_{k}\in\mathbb{R}^{m\times \mathrm{rank}(R_{k})}$ have orthonormal columns spanning $\mathrm{Im}(R_{k})$.
    Then, there exists $\bar{R}_{k}\succ 0$ such that $R_{k}=E_{k}\bar{R}_{k}E_{k}^{\top}$.
    Since a policy having finite KL divergence from $\rho_{k}$ must be supported on $\mathrm{Im}(R_{k})$, the input can be written as $u_{k}=E_{k}\bar{u}_{k}$.
    Hence, Problem \ref{prob:MIODCP} with a fixed degenerate prior $\mathcal{N}(0,R_{k})$ reduces to Problem \ref{prob:MIODCP} with reduced input matrix $B_{k}E_{k}$ and fixed nondegenerate prior $\mathcal N(0,\bar{R}_{k})$.
    If $R_{k}=0$, the input matrix reduces to zero.
    Consequently, Section \ref{subsec:Optimal Policy for Fixed Prior} can be extended to degenerate $\rho \in \mathcal{R}$.
    In particular, Proposition \ref{prop:optimal policy is linear Gaussian} implies that an optimal policy for fixed $\rho \in \mathcal{R}$ belongs to
    \begin{align*}
        \mathcal{P} := &\left\{ \pi=\{\pi_{k}\}_{k=0}^{T-1} \mid  \pi_{k}(\cdot|x)=\mathcal{N}(K_{k}x,V_{k}), \right.\\
        &\left.K_{k}\in \mathbb{R}^{m\times n} ,V_{k}\succeq0,\mathrm{Im}(K_{k})\subset\mathrm{Im}(V_{k})\right\}.
    \end{align*}
    if it exists.
    \label{rem:extension to degenerate prior class}
\end{rem}

\subsection{Optimal Prior for Fixed Policy}
\label{subsec:Optimal Prior for Fixed Policy}

Next, we derive the optimal prior of Problem \ref{prob:MIODCP} with fixed $\pi$.
Based on the results in Section \ref{subsec:Optimal Policy for Fixed Prior}, we fix a policy $\pi$ as $\pi \in \mathcal{P}$.
Since $\pi$ is fixed and $\rho$ does not affect the evolution of $x_{k}$, we have
\begin{align}
    &\inf_{\rho\in \mathcal{R}} J(\pi,\rho)\ \text{s.t. }\eqref{eq:linear system of MIODCP}\text{--}\eqref{eq:boundary conditions of MIODCP}\nonumber\\
    &\Leftrightarrow  \inf_{R_{k}\succeq 0} \mathbb{E}\left[\mathcal{D}_{ \mathrm{KL}}[\pi_{k}(\cdot|x_{k}) \| \mathcal{N}(0,R_{k})]\right], k \in \llbracket 0, T-1 \rrbracket. \label{eq:MIODCP with fixed policy}
\end{align}
Here, we can apply \cite[Proposition 2]{enami2025policy} and the optimal prior for fixed policy is given by the following proposition.
\begin{prop}[{\cite[Proposition 2]{enami2025policy}}]
    Consider a given policy $\pi\in \mathcal{P},\pi_{k}(\cdot|x)=\mathcal{N}(K_{k}x,V_{k})$ and the covariance matrices $\{X_{k}\}_{k=0}^{T}$ given by \eqref{eq:covariance dynamics in convex problem} with $X_{0}=\Sigma_{\mathrm{ini}}$ under $\pi$.
    Then, by denoting $R_{k}^{\pi}:=V_{k} + K_{k} X_{k} K_{k}^{\top}$, the unique optimal prior $\rho^{\pi}$ of Problem \ref{prob:MIODCP} with fixed $\pi$ is given by
    \begin{align}
        &\rho_{k}^{\pi} = \mathcal{N}(0, R_{k}^{\pi}),k \in \llbracket 0, T-1 \rrbracket. \label{eq:optimal prior of MIODCP for fixed policy}
    \end{align}
    \label{prop:optimal prior of MIODCP with policy fixed}
\end{prop}


\subsection{Alternating Optimization Algorithm for MI Optimal Density Control}
\label{subsec:Alternating Optimization Algorithm for MIODCP}

On the basis of Sections \ref{subsec:Optimal Policy for Fixed Prior} and \ref{subsec:Optimal Prior for Fixed Policy}, we propose an alternating optimization algorithm for Problem \ref{prob:MIODCP} as follows:

\begin{alg}
    Initialize prior $\rho^{(0)} \in \mathcal{R}_{+}$ arbitrarily, and iterate the following steps alternately.
    \begin{description}
        \item[P-step:]Calculate the policy $\pi^{(i)} := \pi^{\rho^{(i)}}$. 
        \item[R-step:]Calculate the prior $\rho^{(i+1)} := \rho^{\pi^{(i)}}$.  
    \end{description}
    \label{alg:alternating optimization algorithm for MIODCP}
\end{alg}

Here, we refer to the reason why we initialize Algorithm \ref{alg:alternating optimization algorithm for MIODCP} with a nondegenerate prior $\rho^{(0)}\in \mathcal{R}_{+}$.
From finiteness of the KL term in \eqref{eq:objective function of MIODCP}, it follows that
\begin{align*}
    \mathrm{Im}(R_{k}^{(0)})=\mathrm{Im}(V_{k}^{(0)})=\mathrm{Im}(R_{k}^{(1)})=\cdots,
\end{align*}
where $\pi_{k}^{(i)}(\cdot|x)=\mathcal{N}(K_{k}^{(i)}x,V_{k}^{(i)})$ and $\rho_{k}=\mathcal{N}(0,R_{k}^{(i)})$.
In other words, the image of the covariance matrix is invariant throughout alternating optimization.
Hence, if $R_{k}^{(0)}$ is singular, all subsequent covariance matrices remain singular with the same image and can not converge to a positive definite matrix.
Therefore, to avoid such a restriction on the covariance matrices, we choose $\rho^{(0)}\in \mathcal{R}_{+}$.
This is why we only consider fixed nondegenerate priors in Section \ref{subsec:Optimal Policy for Fixed Prior}.

\begin{rem}
    Let us provide a remark on the feasibility of Algorithm \ref{alg:alternating optimization algorithm for MIODCP}.
    Since $\rho^{\pi}$ can be calculated by Proposition \ref{prop:optimal prior of MIODCP with policy fixed} without any assumptions, the feasibility of Algorithm \ref{alg:alternating optimization algorithm for MIODCP} amounts to the admissibility of Problem \ref{prob:convex problem} for $\rho^{(i)}$ at any finite iteration $i$.
    By noting that $\rho^{(i)}\in \mathcal{R}_{+}$ for any $i \in \mathbb{Z}_{\geq 0}$ with $\rho^{(0)}\in \mathcal{R}_{+}$, Remark \ref{rem:admissibility of convex problem} implies that for any $i \in \mathbb{Z}_{\geq 0}$, Problem \ref{prob:convex problem} is admissible for $\rho^{(i)}$ if it is admissible for $\rho^{(0)}$.
    The latter admissibility condition is equivalent to the reachability of $\Sigma_{\mathrm{fin}}$ from $\Sigma_{\mathrm{ini}}$.
    Noting that for fixed $\rho \in \mathcal{R}_{+}$, the reduction of Problem \ref{prob:MIODCP} to Problem \ref{prob:convex problem} assumes the invertibility of $A_k$, the feasibility of Algorithm \ref{alg:alternating optimization algorithm for MIODCP} ultimately reduces to two conditions: the invertibility of $A_k$ for all $k\in\llbracket 0,T-1\rrbracket$ and the state covariance reachability.
\end{rem}

\subsection{Convergence Analysis}
\label{subsec:Convergence Analysis}

In Section \ref{subsec:Convergence Analysis}, we investigate convergence of Algorithm \ref{alg:alternating optimization algorithm for MIODCP}.
Throughout this subsection, a policy $\pi\in\mathcal{P}$ and a prior $\rho\in\mathcal{R}$ are identified with their parameters $\{K_{k},V_{k}\}_{k=0}^{T-1}$ and $\{R_{k}\}_{k=0}^{T-1}$, respectively.
We write $(\pi^{(i)},\rho^{(i)})\rightarrow (\pi,\rho)$ if $K_{k}^{(i)}\rightarrow K_{k}, V_{k}^{(i)}\rightarrow V_{k}$, and $R_{k}^{(i)}\rightarrow R_{k}$ for any $k\in\llbracket 0,T-1\rrbracket$.
A cluster point is understood with respect to this convergence.

We first consider existence and admissibility of a cluster point of Algorithm \ref{alg:alternating optimization algorithm for MIODCP}.
Before discussing any property of a cluster point of Algorithm \ref{alg:alternating optimization algorithm for MIODCP}, we have to establish its existence.
In addition, even if a cluster point $(\pi^{\circ},\rho^{\circ})$ exists, its admissibility is not trivial.
Indeed, the admissibility of $(\pi^{(i)},\rho^{(i)})$ at any finite iteration does not implies that of $(\pi^{\circ},\rho^{\circ})$ since a cluster point $(\pi^{\circ},\rho^{\circ})$ may be degenerate.
As an extreme case, although each finite iterate $(\pi^{(i)},\rho^{(i)})$ is admissible with $V_{k}^{(i)}\succ0$, it may appear possible that $V_{k}^{(i)}\rightarrow0$ and that a resulting cluster point $(\pi^{\circ},\rho^{\circ})$, whose input support collapses to $\{0\}$, fails to satisfy, for example, the terminal state covariance constraint $X_{T}=\Sigma_{\mathrm{fin}}$.
The following theorem rules out this possibility and guarantees that any cluster point of a sequence of admissible solutions generated by Algorithm \ref{alg:alternating optimization algorithm for MIODCP} remains admissible.

\begin{thm}
    Suppose that $A_{k}$ is invertible for any $k\in\llbracket 0,T-1\rrbracket$ and that Problem \ref{prob:MIODCP} with fixed $\rho^{(0)}$ is admissible.
    Let $\{(\pi^{(i)},\rho^{(i)})\}_{i\in \mathbb{Z}_{\geq 0}}$ be the sequence generated by Algorithm \ref{alg:alternating optimization algorithm for MIODCP}.
    Then the sequence has a cluster point, and any cluster point $(\pi^{\circ},\rho^{\circ})$ is admissible for Problem \ref{prob:MIODCP}.
    \label{thm:existence and admissibility of cluster point}
\end{thm}

\begin{proof}
Note that Algorithm \ref{alg:alternating optimization algorithm for MIODCP} is feasible and $\{(\pi^{(i)},$ $\rho^{(i)})\}_{i\in \mathbb{Z}_{\geq 0}}$ is well-defined from the assumption that $A_{k}$ is invertible for any $k\in\llbracket 0,T-1\rrbracket$ and Problem \ref{prob:MIODCP} with fixed $\rho^{(0)}$ is admissible.

Let $J_{i}:=J(\pi^{(i)},\rho^{(i)})$.
By construction, $J_{i}\geq J(\pi^{(i)},$ $\rho^{(i+1)}) \geq J_{i+1}$.
Since the objective is nonnegative, it follows that
\begin{align}
    0\leq J_i\leq J_{0}\quad \forall i \in \mathbb{Z}_{\geq 0}. \label{eq:bounded objective value in alternating optimization}
\end{align}
For each $i$, denote the second moments under $\pi^{(i)}$ by $X_{k}^{(i)}:=\mathbb{E}[x_{k}x_{k}^{\top}], Y_{k}^{(i)}:=\mathbb{E}[u_{k}x_{k}^{\top}], U_{k}^{(i)}:=\mathbb{E}[u_{k}u_{k}^{\top}]$.
Since the quadratic input cost is a part of the objective, \eqref{eq:bounded objective value in alternating optimization} gives $\frac{1}{2}\sum_{k=0}^{T-1}\mathrm{Tr}[U_{k}^{(i)}] \leq J_{0}$.
Hence, $\{U_{k}^{(i)}\}_{i\in \mathbb{Z}_{\geq 0}}$ is bounded for any $k$.
Suppose that $\{X_{k}^{(i)}\}_{i\in \mathbb{Z}_{\geq 0}}$ is bounded for some $k\in \llbracket0,T-1 \rrbracket$.
Then, the Cauchy-Schwarz inequality yields $\|Y_{k}^{(i)}\|^{2} \leq \mathrm{Tr}[X_{k}^{(i)}]\mathrm{Tr}[U_{k}^{(i)}]$ and thus $\{Y_{k}^{(i)}\}_{i\in \mathbb{Z}_{\geq 0}}$ is also bounded.
Combining boundedness of $\{X_{k}^{(i)}\}_{i\in \mathbb{Z}_{\geq 0}}, \{Y_{k}^{(i)}\}_{i\in \mathbb{Z}_{\geq 0}}$ and $\{U_{k}^{(i)}\}_{i\in \mathbb{Z}_{\geq 0}}$ with the covariance recursion \eqref{eq:covariance dynamics in convex problem}, we obtain that  $\{X_{k+1}^{(i)}\}_{i\in \mathbb{Z}_{\geq 0}}$ is bounded.
By induction starting from $X_{0}^{(i)}=\Sigma_{\mathrm{ini}}$, $\{X_{k}^{(i)}\}_{i\in \mathbb{Z}_{\geq 0}}$ and $\{Y_{k}^{(i)}\}_{i\in \mathbb{Z}_{\geq 0}}$ are bounded for any $k$.
The R-step gives $R_{k}^{(i+1)}=U_{k}^{(i)}$, and hence $\{R_{k}^{(i)}\}_{i\in \mathbb{Z}_{\geq0}}$ is bounded as well.
Therefore, there exists a subsequence such that
\begin{align}
    (X_{k}^{(i_{j})},Y_{k}^{(i_{j})},U_{k}^{(i_{j})},R_{k}^{(i_{j})}) \rightarrow (X_{k}^{\circ},Y_{k}^{\circ},U_{k}^{\circ},R_{k}^{\circ}) \label{eq:convergent subsequence of XYUR}
\end{align}
for any $k$.
Since the positive semidefinite cone is closed, $R_{k}^{\circ}\succeq 0$.
Accordingly, $\rho^{(i_{j})}\in \mathcal{R}_{+}$ converges weakly to $\rho^{\circ}\in \mathcal{R}, \rho_{k}^{\circ}=\mathcal{N}(0,R_{k}^{\circ})$.
Because the covariance evolution \eqref{eq:covariance dynamics in convex problem} is affine in $(X_{k},Y_{k},U_{k})$, taking the limit of
\begin{align*}
    X_{k+1}^{(i_{j})} =& A_{k}X_{k}^{(i_{j})}A_{k}^{\top}+A_{k}Y_{k}^{(i_{j})\top}B_{k}^{\top}+B_{k}Y_{k}^{(i_{j})}A_{k}^{\top} \nonumber\\
    &+B_{k}U_{k}^{(i_{j})}B_{k}^{\top}+W_{k}
\end{align*}
gives
\begin{align*}
    X_{k+1}^{\circ} =& A_{k}X_{k}^{\circ}A_{k}^{\top}+A_{k}Y_{k}^{\circ\top}B_{k}^{\top}+B_{k}Y_{k}^{\circ}A_{k}^{\top} \nonumber\\
                   &+B_{k}U_{k}^{\circ}B_{k}^{\top}+W_{k}
\end{align*}
together with $X_{0}^{\circ}=\Sigma_{\mathrm{ini}}$ and $X_{T}^{\circ}=\Sigma_{\mathrm{fin}}$.

We next show by induction that $X_{k}^{\circ}\succ0$ for any $k\in\llbracket0,T\rrbracket$, while identifying the limiting policy.
Since $X_{0}^{\circ}=\Sigma_{\mathrm{ini}}\succ0$, the claim holds for $k=0$.
Suppose that $X_{k}^{\circ}\succ0$.
Then, by the convergence of the moments,
\begin{align*}
    K_{k}^{(i_{j})} =& Y_{k}^{(i_{j})}(X_{k}^{(i_{j})})^{-1}\rightarrow Y_{k}^{\circ}(X_{k}^{\circ})^{-1} =:K_{k}^{\circ},\\
    V_{k}^{(i_{j})} =& U_{k}^{(i_{j})}-Y_{k}^{(i_{j})}(X_{k}^{(i_{j})})^{-1}(Y_{k}^{(i_{j})})^{\top}\\
    &\rightarrow U_{k}^{\circ}- Y_{k}^{\circ}(X_{k}^{\circ})^{-1}(Y_{k}^{\circ})^{\top}=:V_{k}^{\circ},
\end{align*}
Hence, $\pi_{k}^{(i_{j})}(\cdot|x)\rightarrow \pi_{k}^{\circ}(\cdot|x)=\mathcal{N}(K_{k}^{\circ}x,V_{k}^{\circ})$.
Since $\{J_{i}\}_{i\in \mathbb{Z}_{\geq 0}}$ is nonincreasing and all stage costs are nonnegative, the KL cost satisfies $\mathbb{E}[\mathcal{D}_{\mathrm{KL}}[\pi_{k}^{(i_{j})}(\cdot|x_{k})\|\rho_{k}^{(i_{j})}]]\leq J_{0}$ for any $j$.
This KL cost can be calculated as
\begin{align}
    &2\mathbb{E}\left[\mathcal{D}_{\mathrm{KL}}\left[\pi_{k}^{(i_{j})}(\cdot|x_{k})\|\rho_{k}^{(i_{j})}\right]\right]\nonumber\\
    &=\mathrm{Tr}\left[(R_{k}^{(i_{j})})^{-1}V_{k}^{(i_{j})}\right]-
    \log|(R_{k}^{(i_{j})})^{-1}V_{k}^{(i_{j})}|-m \nonumber\\
    &\quad+\mathrm{Tr}\left[(R_{k}^{(i_{j})})^{-1} K_{k}^{(i_{j})}X_{k}^{(i_{j})}(K_{k}^{(i_{j})})^{\top}\right]. \label{eq:bounded Gaussian KL}
\end{align}
The last term in \eqref{eq:bounded Gaussian KL} is nonnegative.
In addition, letting $S_{k}^{(i_{j})}:=(R_{k}^{(i_{j})})^{-1/2}V_{k}^{(i_{j})}(R_{k}^{(i_{j})})^{-1/2}\succ0$ and denoting its eigenvalues by $\lambda_{k,\ell}^{(i_{j})}>0$, $\ell\in\llbracket1,m\rrbracket$, we have
\begin{align}
    &\mathrm{Tr}\left[(R_{k}^{(i_{j})})^{-1}V_{k}^{(i_{j})}\right] -\log\left|(R_{k}^{(i_{j})})^{-1}V_{k}^{(i_{j})}\right|-m
    \nonumber\\
    &=\sum_{\ell=1}^{m}\left(\lambda_{k,\ell}^{(i_{j})}-\log\lambda_{k,\ell}^{(i_{j})}-1\right).
\end{align}
Since $\lambda-\log\lambda-1 \geq0$ diverges to $\infty$ as $\lambda\rightarrow0$ or $\lambda \rightarrow \infty$, the bound in \eqref{eq:bounded Gaussian KL} implies that there exist constants $0<\underline{c}_{k}\leq \bar{c}_{k}<\infty$, independent of $j$, such that $\underline{c}_{k}$
\begin{align}
    &\underline{c}_{k} \leq \lambda_{k,\ell}^{(i_{j})} \leq \bar{c}_{k}\ \forall \ell \in \llbracket1,m\rrbracket\nonumber\\
    &\Leftrightarrow  \underline{c}_{k}R_{k}^{(i_{j})} \preceq V_{k}^{(i_{j})} \preceq \bar{c}_{k}R_{k}^{(i_{j})} \label{eq:RV finite comparison}
\end{align}
In addition, \eqref{eq:bounded Gaussian KL} yields $\mathrm{Tr}\left[(R_{k}^{(i_{j})})^{-1}K_{k}^{(i_{j})}X_{k}^{(i_{j})}(K_{k}^{(i_{j})})^{\top}\right]$ $\leq 2J_{0}$.
Since $(R_{k}^{(i_{j})})^{-1/2}K_{k}^{(i_{j})}X_{k}^{(i_{j})}(K_{k}^{(i_{j})})^{\top}(R_{k}^{(i_{j})})^{-1/2}\succeq0$, its largest eigenvalue is bounded by its trace.
Hence,
\begin{align}
    &(R_{k}^{(i_{j})})^{-1/2} K_{k}^{(i_{j})}X_{k}^{(i_{j})} (K_{k}^{(i_{j})})^{\top} (R_{k}^{(i_{j})})^{-1/2} \preceq 2J_{0}I\nonumber \\
    &\Leftrightarrow K_{k}^{(i_{j})}X_{k}^{(i_{j})}(K_{k}^{(i_{j})})^{\top}\preceq 2J_{0}R_{k}^{(i_{j})}. \label{eq:KXR finite comparison}
\end{align}
Taking the limit in \eqref{eq:RV finite comparison} and \eqref{eq:KXR finite comparison} gives
\begin{align*}
    \underline{c}_{k}R_{k}^{\circ}  &\preceq  V_{k}^{\circ} \preceq  \bar{c}_{k}R_{k}^{\circ},\\
    K_{k}^{\circ}X_{k}^{\circ}(K_{k}^{\circ})^{\top} &\preceq 2J_{0}R_{k}^{\circ}.
\end{align*}
Consequently, $\mathrm{Im}(V_{k}^{\circ})=\mathrm{Im}(R_{k}^{\circ}), \mathrm{Im}(K_{k}^{\circ}) \subset \mathrm{Im}(R_{k}^{\circ})$.
The recursion of $X_{k}^{\circ}$ can now be written as
\begin{align}
    X_{k+1}^{\circ} =& (A_{k}+B_{k}K_{k}^{\circ}) X_{k}^{\circ} (A_{k}+B_{k}K_{k}^{\circ})^{\top} \nonumber\\
    &+ B_{k}V_{k}^{\circ}B_{k}^{\top} +W_{k}. \label{eq:limit covariance recursion KV}
\end{align}
Suppose that $z^{\top}X_{k+1}^{\circ}z=0$ for some $z\neq 0$.
Since all terms on the right-hand side of \eqref{eq:limit covariance recursion KV} are positive semidefinite,
\begin{align}
    (A_{k}+B_{k}K_{k}^{\circ})^{\top}z=0,B_{k}^{\top}z\in\ker(V_{k}^{\circ}).
\end{align}
By $\mathrm{Im}(K_{k}^{\circ}) \subset \mathrm{Im}(R_{k}^{\circ})=\mathrm{Im}(V_{k}^{\circ})$, we have $(K_{k}^{\circ})^{\top}B_{k}^{\top}z=0$, and hence $A_{k}^{\top}z=0$.
Since $A_{k}$ is invertible, $z=0$, which contradicts the choice of $z\neq 0$.
Therefore, $X_{k+1}^{\circ}\succ0$.
By induction, $X_{k}^{\circ}\succ 0$ for any $k \in \llbracket0,T\rrbracket$.

Consequently, $\pi^{\circ}$ is well-defined and the state covariance matrices under $\pi^{\circ}$ coincide with $\{X_{k}^{\circ}\}_{k=0}^{T}$.
Since the initial state, the policy, and the system noise are Gaussian, the state remains Gaussian under $\pi^{\circ}$.
Hence, $\pi^{\circ}$ satisfies the boundary constraint \eqref{eq:boundary conditions of MIODCP}, and
$(\pi^{\circ},\rho^{\circ})$ is feasible for Problem \ref{prob:MIODCP}.
Finally, lower semicontinuity of the KL term and continuity of the quadratic input cost imply
\begin{align*}
    J(\pi^{\circ},\rho^{\circ}) \leq \liminf_{j\rightarrow\infty}J(\pi^{(i_{j})},\rho^{(i_{j})}) \leq J_{0}<\infty.
\end{align*}
Thus, $(\pi^{\circ},\rho^{\circ})$ is admissible for Problem \ref{prob:MIODCP}.
Since the argument applies to an arbitrary cluster point, the proof is complete.
\end{proof}

We finally establish coordinate-wise optimality of a cluster point under an additional assumption.

\begin{thm}
    Suppose the assumptions of Theorem \ref{thm:existence and admissibility of cluster point}, and let $(\pi^{\circ},\rho^{\circ})$ be a cluster point of the sequence $\{(\pi^{(i)},\rho^{(i)})\}_{i\in \mathbb{Z}_{\geq 0}}$ generated by Algorithm \ref{alg:alternating optimization algorithm for MIODCP}.
    Then,
    \begin{align}
        \rho^{\circ} \in \arg\min_{\rho\in \mathcal{R}} J(\pi^{\circ},\rho). \label{eq:coordinate wise optimal prior}
    \end{align}
    In addition, if $\rho^{\circ}\in \mathcal{R}_{+}$, then
    \begin{align}
        \pi^{\circ} \in \arg\min_{\pi} J(\pi,\rho^{\circ}). \label{eq:coordinate wise optimal policy}
    \end{align}
    \label{thm:coordinatewise optimality in MIODCP}
\end{thm}

\begin{proof}
Let $J_{i}:=J(\pi^{(i)},\rho^{(i)})$.
By optimality of P- and R-steps, $J_{i}\geq J(\pi^{(i)},\rho^{(i+1)}) \geq J_{i+1}$.
Since $J_{i}\geq0$, the sequence $\{J_{i}\}_{i\in \mathbb{Z}_{\geq 0}}$ converges, and hence
\begin{align}
    J_{i}-J_{i+1}\rightarrow0. \label{eq:objective difference converges to zero}
\end{align}
We first establish optimality with respect to the prior.
Noting that $\rho_{k}^{(i+1)}$ is equal to the marginal distribution of $u_{k}$ under $\pi_{k}^{(i)}$ from Proposition \ref{prop:optimal prior of MIODCP with policy fixed}, we obtain
\begin{align}
    J(\pi^{(i)},\rho^{(i)}) - J(\pi^{(i)},\rho^{(i+1)})=\sum_{k=0}^{T-1} \mathcal{D}_{\mathrm{KL}}[\rho_{k}^{(i+1)}\|\rho_{k}^{(i)}].\nonumber
\end{align}
Combining this with \eqref{eq:objective difference converges to zero} yields $\mathcal{D}_{\mathrm{KL}}[\rho_{k}^{(i+1)}\|\rho_{k}^{(i)}] \rightarrow 0$ for any $k$.
By Pinsker's inequality \cite{cover2006elements}, $\|\rho_{k}^{(i+1)}-\rho_{k}^{(i)}\|_{\mathrm{TV}}\rightarrow0$, where $\|\cdot\|_{\mathrm{TV}}$ denotes the total variation distance.
Hence, along the subsequence $\{i_{j}\}$, $\rho^{(i_{j}+1)}\rightarrow\rho^{\circ}$.
Since $\pi^{(i_{j})}\rightarrow\pi^{\circ}$ and $\rho^{\pi}$ is continuous in the parameters of $\pi \in \mathcal{P}$ from Proposition \ref{prop:optimal prior of MIODCP with policy fixed}, $\rho^{(i_{j}+1)}=\rho^{\pi^{(i_{j})}} \rightarrow \rho^{\pi^{\circ}}$.
Therefore, $\rho^{\circ}=\rho^{\pi^{\circ}}$, which proves \eqref{eq:coordinate wise optimal prior}.

We next establish optimality with respect to the policy under the assumption of $\rho^{\circ}\in \mathcal{R}_{+}$.
Let $\hat{\pi}$ be an arbitrary policy admissible for Problem \ref{prob:MIODCP} with fixed $\rho^{\circ}$.
Then, from Proposition \ref{prop:optimal policy is linear Gaussian}, there exists an admissible Gaussian policy $\tilde{\pi}\in \mathcal{P}$ such that $J(\tilde{\pi},\rho^{\circ})\leq J(\hat{\pi},\rho^{\circ})<\infty$.
Note that the assumption $\rho^{\circ} \in \mathcal{R}_{+}$ implies $\tilde{\pi} \in \mathcal{P}_{+}$.
For fixed admissible policy $\tilde{\pi} \in \mathcal{P}_{+}$, $J(\tilde{\pi},\rho)$ is continuous in the parameters $\{R_{k}\}_{k=0}^{T-1}$ of nondegenerate $\rho \in \mathcal{R}_{+}$.
Therefore,
\begin{align}
    \lim_{j\rightarrow \infty}J(\tilde{\pi},\rho^{(i_{j})})=J(\tilde{\pi},\rho^{\circ}). \label{eq:objective value continuity in convergence with fixed policy}
\end{align}
In addition, $J(\pi,\rho)$ is continuous in the parameters $\{(K_{k},V_{k})\}_{k=0}^{T-1}$ and $\{R_{k}\}_{k=0}^{T-1}$ of nondegenerate $\pi \in \mathcal{P}_{+}$ and $\rho \in \mathcal{R}_{+}$.
Furthermore, since $\rho^{\circ}\in \mathcal{R}_{+}$ and Theorem \ref{thm:existence and admissibility of cluster point} implies admissibility of $(\pi^{\circ}, \rho^{\circ})$, it follows that $\pi^{\circ}\in \mathcal{P}_{+}$.
Combining this with that $(\pi^{(i_{j})},\rho^{(i_{j})})\in \mathcal{P}_{+} \times \mathcal{R}_{+}$ for any $j \in \mathbb{Z}_{\geq 0}$, it follows that 
\begin{align}
    \lim_{j\rightarrow \infty}J(\pi^{(i_{j})},\rho^{(i_{j})})= J(\pi^{\circ},\rho^{\circ}). \label{eq:objective value continuity in convergence}
\end{align}
Then, the optimality $J(\pi^{(i_{j})},\rho^{(i_{j})})\leq J(\tilde{\pi},\rho^{(i_{j})})$ and the fact $J(\tilde{\pi},\rho^{\circ})\leq J(\hat{\pi},\rho^{\circ})$ imply $J(\pi^{\circ},\rho^{\circ}) \leq J(\tilde{\pi},\rho^{\circ})\leq J(\hat{\pi},\rho^{\circ})$. 
Since $\hat{\pi}$ was an arbitrary admissible policy, \eqref{eq:coordinate wise optimal policy} follows, which completes the proof.
\end{proof}

The coordinate-wise optimality conditions \eqref{eq:coordinate wise optimal prior} and \eqref{eq:coordinate wise optimal policy} require different assumptions.
The optimality of $\rho^{\circ}$ in \eqref{eq:coordinate wise optimal prior} follows solely from the admissibility assumption of Algorithm \ref{alg:alternating optimization algorithm for MIODCP} since Proposition \ref{prop:optimal prior of MIODCP with policy fixed} guarantees continuity of $\pi\mapsto\rho^{\pi}$ on $\mathcal{P}$.
In contrast, the proof of \eqref{eq:coordinate wise optimal policy} relies on continuity of $J(\pi,\rho)$.
Although $J(\pi,\rho)$ is continuous on $\mathcal{P}_{+}\times \mathcal{R}_{+}$, the continuity is not guaranteed at boundaries of the whole set $\mathcal{P}\times  \mathcal{R}$ where a nondegenerate solution becomes degenerate.
Accordingly, Theorem \ref{thm:coordinatewise optimality in MIODCP} additionally assumes $\rho^{\circ}\in\mathcal{R}_{+}$ to avoid the subtle issue of the boundaries.

\section{Formulation of Generalized Schr\"{o}dinger Bridges with Reference Refinement}
\label{sec:Formulation of Generalized Schrodinger Bridges}

In this section, we formulate a generalized Schrödinger bridge problem with reference refinement associated with the linear system.
In contrast to a conventional Schrödinger bridge with a fixed reference process, both the controlled process and the reference process are optimized in the proposed formulation.

Consider the following linear system:
\begin{align}
    x_{k+1}=A_{k}x_{k}+B_{k}w_{k}^{\rho}, w_{k}^{\rho}\sim \rho_{k}, x_{0}\sim \mathcal{N}(0,\Sigma_{ \mathrm{ini}}), \label{eq:reference process}
\end{align}
where $\Sigma_{ \mathrm{ini}}\succ 0, \rho \in \mathcal{R}, \rho_{k}=\mathcal{N}(0,R_{k})$.
The noises $\{w_{k}^{\rho}\}_{k=0}^{T-1}$ are mutually independent and independent of $x_{0}$.
While $\rho_{k}$ was introduced as a prior in MI optimal density control, in this section, we treat it as the distribution of the noise $w_{k}^\rho$ in the system \eqref{eq:reference process} and assume that its covariance matrix $R_{k}$ is unknown.
In addition, assume that snapshot data at time $k=0$ and $k=T$ are given by $\mathcal{N}(0,\Sigma_{ \mathrm{ini}})$ and $\mathcal{N}(0,\Sigma_{ \mathrm{fin}})$, respectively.
To construct a reference covariance sequence from the snapshot data, we consider a Schr\"{o}dinger bridge problem with reference refinement.
To this end, let us introduce some notations.
The state process $\{x_{k}\}_{k=0}^{T}$ given by \eqref{eq:reference process} is called a reference process.
Denote by $\mathbb{Q}^{\rho}$ the probability distribution on $\chi:=(\mathbb{R}^{n})^{T+1}$ of the reference process $\{x_{k}\}_{k=0}^{T}$ given by \eqref{eq:reference process}.
Define as
\begin{align*}
    \Pi(\mathcal{R}):=\{\mathbb{Q}^{\rho} \mid \rho \in \mathcal{R}\}
\end{align*}
the set of all reference process distributions generated by $\rho \in \mathcal{R}$.
Denote by
\begin{align*}
    \Pi(\Sigma_{ \mathrm{ini}},\Sigma_{ \mathrm{fin}})
\end{align*}
the set of all probability distributions on $\chi$ whose marginals at $k=0$ and $k=T$ are equal to $\mathcal{N}(0,\Sigma_{ \mathrm{ini}})$ and $\mathcal{N}(0,\Sigma_{ \mathrm{fin}})$, respectively.
In addition, for a probability distribution $\mathbb{P}$ on $\chi$, denote the distribution of $x_{k}$ and the conditional distribution of $x_{k+1}$ given $x_{k}$ by $\mathbb{P}_{k}$ and $\mathbb{P}_{k+1|k}$, respectively.
Using these notations, the Schr\"{o}dinger bridge problem is formulated as follows.
\begin{prob}
     Find a pair of a controlled process distribution $\mathbb{P} \in \Pi(\Sigma_{ \mathrm{ini}},\Sigma_{ \mathrm{fin}})$ and a reference process distribution $\mathbb{Q}^{\rho} \in \Pi(\mathcal{R})$ that solves
     \begin{align}
         \inf_{\mathbb{P},\mathbb{Q}^{\rho}} &\hat{J}(\mathbb{P},\mathbb{Q}^{\rho}):=\mathcal{D}_{ \mathrm{KL}}\left[\mathbb{P}\|\mathbb{Q}^{\rho}\right] + \mathbb{E}_{\mathbb{P}}[V(x_{0},\ldots,x_{T})],\label{eq:objective function of SBP}
     \end{align}
     where
     \begin{align}
        V_{k}(x_{k},x_{k+1}):=&\frac{1}{2}\|x_{k+1}-A_{k}x_{k}\|_{B_{k}^{\dagger \top}B_{k}^{\dagger}}^{2}, \label{eq:decomposed potential energy in SBP}\\
         V(x_{0},\ldots,x_{T}):=&\sum_{k=0}^{T-1}V_{k}(x_{k},x_{k+1}) .\label{eq:potential energy in SBP}
     \end{align}
      
     \label{prob:SBP}
\end{prob}

If we fix the reference process distribution $\mathbb{Q}^{\rho}$ and ignore the second term in \eqref{eq:objective function of SBP}, Problem \ref{prob:SBP} reduces to the conventional Schr\"{o}dinger bridge problem.
The optimization of $\mathbb{Q}^{\rho}$ corresponds to the reference refinement, which serves to estimate $\{R_{k}\}_{k=0}^{T-1}$.
In addition, the second term in \eqref{eq:objective function of SBP} renders Problem \ref{prob:SBP} a generalized Schr\"{o}dinger bridge problem, which incorporates desired specifications of the state.
For example, if we set $V(x_{0},\ldots,x_{T})=\sum_{k=0}^{T}\|x_{k}\|^{2}$ instead of \eqref{eq:potential energy in SBP}, the regularization term reflects prior knowledge that the state in the controlled process tends to stay near zero.
In the case with \eqref{eq:potential energy in SBP}, assuming that the controlled process is given in the form $x_{k+1}=A_{k}x_{k}+B_{k}u_{k}$, the regularization term implies that the controlled process is realized with as little input energy as possible.

Note that since the endpoint snapshots do not in general uniquely identify the intermediate noise covariances, the estimation of $\{R_{k}\}_{k=0}^{T-1}$ selected by Problem \ref{prob:SBP} should be understood as a regularized construction of a reference covariance sequence explaining the endpoint snapshots, rather than as identification of the true noise covariance sequence.

\begin{rem}
    Similar to Problem \ref{prob:MIODCP}, Problem \ref{prob:SBP} can be extended to the case where the marginal distributions are given by nonzero mean Gaussian distributions $x_{0}\sim \mathcal{N}(\mu_{ \mathrm{ini}},\Sigma_{\mathrm{ini}}), x_{T}\sim \mathcal{N}(\mu_{\mathrm{fin}},\Sigma_{\mathrm{fin}})$.
    For the details, see Appendix \ref{app:Generalized Schrodinger Bridges with Nonzero Margianl Means}.
    \label{rem:GSB with nonzero mean marginal constraints}
\end{rem}

We close this section by introducing the following lemma.

\begin{lem}
    For any prior $\rho \in \mathcal{R}$ and any admissible controlled process $\mathbb{P}$ for fixed $\mathbb{Q}^{\rho}$, there exists an admissible Markov controlled process distribution $\mathbb{P}^{\prime}$ satisfying $\hat{J}(\mathbb{P}^{\prime},\mathbb{Q}^{\rho})\leq \hat{J}(\mathbb{P},\mathbb{Q}^{\rho})$, where equality holds if and only if $\mathbb{P}$ is Markov.
    \label{lem:Markov solution of SBP}
\end{lem}

\begin{proof}
    Noting that $\mathbb{Q}^{\rho}$ is Markov from \eqref{eq:reference process}, we can decompose $\mathbb{P}$ and $\mathbb{Q}^{\rho}$ as
    \begin{align*}
        \mathbb{P}(dx_{0},\ldots,dx_{T}) =& \mathbb{P}_{0}(dx_{0}) \prod_{k=0}^{T-1} \mathbb{P}(dx_{k+1} | x_{0},\ldots,x_{k}),\\
        \mathbb{Q}^{\rho}(dx_{0},\ldots,dx_{T}) =& \mathbb{Q}_{0}^{\rho}(dx_0) \prod_{k=0}^{T-1} \mathbb{Q}_{k+1|k}^{\rho}(dx_{k+1} | x_{k}).
    \end{align*}
    Using this expression, we have
    \begin{align}
        &\mathcal{D}_{\mathrm{KL}}[\mathbb{P}\|\mathbb{Q}^{\rho}]\nonumber\\
        &= \mathcal{D}_{\mathrm{KL}}[\mathbb{P}_{0}\|\mathbb{Q}_{0}^{\rho}]\nonumber \\
        &\quad+  \sum_{k=0}^{T-1} \mathbb{E}_{\mathbb{P}_{k}}[\mathcal{D}_{\mathrm{KL}}[\mathbb{P}_{k+1|k}(\cdot|x_{k})\|\mathbb{Q}_{k+1|k}^{\rho}(\cdot|x_{k})]]\nonumber\\
        &\quad+  \sum_{k=0}^{T-1} \mathbb{E}_{\mathbb{P}}[\mathcal{D}_{\mathrm{KL}}[\mathbb{P}(\cdot|x_{0},\ldots,x_{k})\|\mathbb{P}_{k+1|k}(\cdot|x_{k})]]. \nonumber 
    \end{align}
    Using $\mathbb{P}$, we construct a Markov distribution $\mathbb{P}^{\prime}$ such that
    \begin{align*}
        \mathbb{P}^{\prime}(dx_{0},\ldots,dx_{T})=\mathbb{P}_{0}(dx_{0})\prod_{k=0}^{T-1}\mathbb{P}_{k+1|k}(dx_{k+1}|x_{k}).
    \end{align*}
    Noting that $\mathbb{P}_{0}=\mathbb{P}^{\prime}_{0}$ and $\mathbb{P}(dx_{k},dx_{k+1})=\mathbb{P}^{\prime}(dx_{k},dx_{k+1})$, we have
    \begin{align*}
        &\mathcal{D}_{\mathrm{KL}}\left[\mathbb{P}\|\mathbb{Q}^{\rho}\right] + \mathbb{E}_{\mathbb{P}}[V(x_{0},\ldots,x_{T})]\\
        &= \mathcal{D}_{\mathrm{KL}}[\mathbb{P}_{0}\|\mathbb{Q}_{0}^{\rho}] \\
        &\quad+  \sum_{k=0}^{T-1} \mathbb{E}_{\mathbb{P}_{k}}[\mathcal{D}_{\mathrm{KL}}[\mathbb{P}_{k+1|k}(\cdot|x_{k})\|\mathbb{Q}_{k+1|k}^{\rho}(\cdot|x_{k})]]\nonumber\\
        &\quad+  \sum_{k=0}^{T-1} \mathbb{E}_{\mathbb{P}}[\mathcal{D}_{\mathrm{KL}}[\mathbb{P}(\cdot|x_{0},\ldots,x_{k})\|\mathbb{P}_{k+1|k}(\cdot|x_{k})]]\\
        &\quad+\sum_{k=0}^{T-1}\mathbb{E}_{\mathbb{P}}[V_{k}(x_{k},x_{k+1})]\\
        &\geq  \mathcal{D}_{\mathrm{KL}}[\mathbb{P}_{0}^{\prime}\|\mathbb{Q}_{0}^{\rho}]\\
        &\quad+  \sum_{k=0}^{T-1} \mathbb{E}_{\mathbb{P}_{k}^{\prime}}[\mathcal{D}_{\mathrm{KL}}[\mathbb{P}_{k+1|k}^{\prime}(\cdot|x_{k})\|\mathbb{Q}_{k+1|k}^{\rho}(\cdot|x_{k})]]\nonumber\\
        &\quad +\sum_{k=0}^{T-1}\mathbb{E}_{\mathbb{P}^{\prime}}[V_{k}(x_{k},x_{k+1})]\\
        &=\mathcal{D}_{\mathrm{KL}}\left[\mathbb{P}^{\prime}\|\mathbb{Q}^{\rho}\right] + \mathbb{E}_{\mathbb{P}^{\prime}}[V(x_{0},\ldots,x_{T})],
    \end{align*}
    where equality holds if and only if $\mathbb{P}$ is Markov.
    Therefore, the claim of Lemma \ref{lem:Markov solution of SBP} holds.
\end{proof}

On the basis of Lemma \ref{lem:Markov solution of SBP}, we only consider Markov controlled processes henceforth.

\section{Alternating Optimization for Generalized Schr\"{o}dinger Bridges} \label{sec:Alternating Optimization for Schrodinger Bridges}

Before deriving an alternating optimization algorithm for Problem \ref{prob:SBP}, we first clarify its relation to Problem \ref{prob:MIODCP} without process noise $\{w_{k}\}_{k=0}^{T-1}$.
\begin{prop}
    Suppose that $B_{k}$ has full column rank and $W_{k}=0$ for any $k\in\llbracket0,T-1\rrbracket$.
    Then, for any admissible pair $(\pi,\rho)$ of Problem \ref{prob:MIODCP}, the corresponding process distributions $\mathbb{P}^{\pi}$ and $\mathbb{Q}^{\rho}$ satisfy
    \begin{align}
        &\hat{J}(\mathbb{P}^{\pi},\mathbb{Q}^{\rho})=J(\pi,\rho), \label{eq:objective function equivalence}
    \end{align}
    where $\mathbb{P}^{\pi}$ is the probability distribution of the following process:
    \begin{align}
        x_{k+1}=A_{k}x_{k}+B_{k}u_{k}, u_{k}\sim \pi_{k}(\cdot|x_{k}),x_{0}\sim \mathcal{N}(0,\Sigma_{ \mathrm{ini}}). \label{eq:controlled process driven by policy}
    \end{align}
    Conversely, for any admissible Markov pair $(\mathbb{P},\mathbb{Q}^{\rho})$ of Problem \ref{prob:SBP}, there exists a policy $\pi$ such that $\mathbb{P}=\mathbb{P}^{\pi}$ and $(\pi,\rho)$ is admissible for Problem \ref{prob:MIODCP}.
    \label{prop:relationship between MIODCP and GSB}
\end{prop}

\begin{proof}
Since $B_{k}$ has full column rank, for any fixed $x\in\mathbb{R}^{n}$, the mapping
\begin{align*}
    u \mapsto x_{k+1}=A_{k}x+B_{k}u
\end{align*}
is one-to-one and its inverse is given by $u=B_{k}^{\dagger}(x_{k+1}-A_{k}x)$.
By the invariance of the KL divergence under one-to-one transformations, we have
\begin{align*}
    &\mathcal{D}_{\mathrm{KL}}\left[\mathbb{P}_{k+1|k}^{\pi}(\cdot|x)\|\mathbb{Q}_{k+1|k}^{\rho}(\cdot|x)\right] = \mathcal{D}_{\mathrm{KL}} \left[ \pi_{k}(\cdot|x) \| \rho_{k} \right].
\end{align*}
Since both $\mathbb{P}^{\pi}$ and $\mathbb{Q}^{\rho}$ are Markov,
the chain rule for the KL divergence and $\mathbb{P}^{\pi}_{0}=\mathbb{Q}_{0}^{\rho}=\mathcal{N}(0,\Sigma_{\mathrm{ini}})$ give
\begin{align}
    \mathcal{D}_{\mathrm{KL}}[\mathbb{P}^{\pi}\Vert\mathbb{Q}^{\rho}]=\sum_{k=0}^{T-1}\mathbb{E}_{\mathbb{P}^{\pi}}\left[\mathcal{D}_{\mathrm{KL}}[\pi_{k}(\cdot|x_{k})\|\rho_{k}]\right].  \label{eq:KL term correspondence}
\end{align}
In addition, since $x_{k+1}-A_{k}x_{k}=B_{k}u_{k}$ and $B_{k}^{\dagger}B_{k}=I$,
\begin{align}
    V_{k}(x_{k},x_{k+1}) = \frac{1}{2}\|u_{k}\|^{2}. \label{eq:input energy correspondence}
\end{align}
Combining \eqref{eq:KL term correspondence} and \eqref{eq:input energy correspondence}, we obtain \eqref{eq:objective function equivalence}.

Conversely, let $(\mathbb{P},\mathbb{Q}^{\rho})$ be any admissible Markov pair of Problem \ref{prob:SBP}.
Let $\pi_{k}^{\mathbb{P}}(\cdot|x_{k})$ be the conditional distribution of
\begin{align*}
    u_{k}:=B_{k}^{\dagger}(x_{k+1}-A_{k}x_{k})
\end{align*}
given $x_{k}$ under $\mathbb{P}$.
Since $x_{k+1}-A_{k}x_{k}\in\mathrm{Im}(B_{k})$ $\mathbb{P}$-a.s. due to $\mathcal{D}_{\mathrm{KL}}[\mathbb{P}\Vert\mathbb{Q}^{\rho}]<\infty$  and $B_{k}B_{k}^{\dagger}$ is the orthogonal projection onto $\mathrm{Im}(B_{k})$, we have $x_{k+1}=A_{k}x_{k}+B_{k}u_{k}$ $\mathbb{P}$-a.s.
Therefore, $\mathbb{P}=\mathbb{P}^{\pi^{\mathbb{P}}}$.
Since $(\mathbb{P},\mathbb{Q}^{\rho})$ is admissible, applying \eqref{eq:KL term correspondence} and \eqref{eq:input energy correspondence} gives
\begin{align*}
    J(\pi^{\mathbb{P}},\rho)=\hat{J}(\mathbb{P},\mathbb{Q}^{\rho}) < \infty.
\end{align*}
Therefore, $\pi^{\mathbb{P}}$ is admissible for Problem \ref{prob:MIODCP} with fixed $\rho$, which completes the proof.
\end{proof}

\begin{rem}
    The assumption of full column rank $B_{k}$ ensures that, for any fixed $x_{k}\in\mathbb{R}^{n}$, the mapping $u_{k}\mapsto x_{k+1}=A_{k}x_{k}+B_{k}u_{k}$ is one-to-one.
    This property is essential for the exact correspondence between the objective functions of Problems \ref{prob:MIODCP} and \ref{prob:SBP} established in Proposition \ref{prop:relationship between MIODCP and GSB}.
    This assumption excludes redundant input directions and allows the input to be uniquely recovered from the state transition.
    \label{rem:validity of assumptions of Theorem of optimal bridge process}
\end{rem}

Proposition \ref{prop:relationship between MIODCP and GSB} establishes a correspondence between the objective functions of Problems \ref{prob:MIODCP} and \ref{prob:SBP} only for controlled processes expressed by \eqref{eq:controlled process driven by policy}.
However, Problem \ref{prob:SBP} is optimized over a more general class of controlled process distributions.
In the following subsections, we reveal that the alternating optimization procedures in Problems \ref{prob:MIODCP} and \ref{prob:SBP} coincide under $W_{k}=0$ and full column rank $B_{k}$ while showing that we can restrict controlled processes to the form \eqref{eq:controlled process driven by policy} without loss of generality.

\subsection{Optimal Controlled Process for Fixed Reference Process} \label{subsec:Optimal Bridge Process for Fixed Reference Process}

In this subsection, we only consider the following reference process class
\begin{align}
    \Pi(\mathcal{R}_{+}):=\{\mathbb{Q}^{\rho} \mid \rho \in \mathcal{R}_{+}\} \nonumber
\end{align}
instead of $\Pi(\mathcal{R})$ for the same reason as why we only consider the nondegenerate prior class $\mathcal{R}_{+}$ in Section \ref{subsec:Optimal Policy for Fixed Prior}.
We show that the optimal state process of Problem \ref{prob:MIODCP} with the prior $\rho$ fixed is actually the optimal controlled process of Problem \ref{prob:SBP} with the reference process distribution $\mathbb{Q}^{\rho}$ fixed.

\begin{prop}
    Consider a given prior $\rho \in \mathcal{R}_{+}, \rho_{k}=\mathcal{N}(0,R_{k})$.
    Suppose the assumptions of Proposition \ref{prop:existence and uniqueness of convex problem}.
    In addition, suppose that $B_{k}$ is full column rank and $W_{k}=0$ for any $k\in \llbracket0,T-1 \rrbracket$.
    Then, the unique optimal controlled process distribution of Problem \ref{prob:SBP} with fixed reference process distribution $\mathbb{Q}^{\rho}$ is given by $\mathbb{P}^{\pi^{\rho}}$.
    \label{prop:optimal bridge process for fixed reference process}
\end{prop}

\begin{proof}
    Let $\mathbb{P}$ be an arbitrary admissible controlled process distribution of Problem \ref{prob:SBP} with fixed $\mathbb{Q}^{\rho}$.
    By Lemma \ref{lem:Markov solution of SBP}, there exists a Markov process distribution $\mathbb{P}^{\prime}$ satisfying $\hat{J}(\mathbb{P}^{\prime},\mathbb{Q}^{\rho})\leq \hat{J}(\mathbb{P},\mathbb{Q}^{\rho})$.
    By Proposition \ref{prop:relationship between MIODCP and GSB}, there exists an admissible policy $\pi^{\prime}$ such that $\mathbb{P}^{\prime}=\mathbb{P}^{\pi^{\prime}}$ and $\hat{J}(\mathbb{P}^{\prime},\mathbb{Q}^{\rho})=J(\pi^{\prime},\rho)$.
    Since $\mathbb{P}^{\pi^{\rho}}$ is admissible, the optimality $\pi^{\rho}$ implies
    \begin{align}
         \hat{J}(\mathbb{P}^{\pi^{\rho}},\mathbb{Q}^{\rho})
         &=J(\pi^{\rho},\rho)\nonumber\\
         &\leq J(\pi^{\prime},\rho)=\hat{J}(\mathbb{P}^{\prime},\mathbb{Q}^{\rho})\leq \hat{J}(\mathbb{P},\mathbb{Q}^{\rho}). \label{eq:inequalities for optimal controlled process}
    \end{align}
    Since $\mathbb{P}$ is arbitrary, $\mathbb{P}^{\pi^{\rho}}$ is an optimal controlled process distribution of Problem \ref{prob:SBP} with fixed $\mathbb{Q}^{\rho}$.
    From Corollary \ref{cor:uniqueness of optimal policy for fixed prior}, the first inequality in \eqref{eq:inequalities for optimal controlled process} holds only if $\pi^{\rho}$ coincides with $\pi^{\prime}$ in the sense specified therein.
    Therefore, $\mathbb{P}^{\pi^{\rho}}$ is the unique optimal controlled process distribution.
\end{proof}

From Proposition \ref{prop:optimal bridge process for fixed reference process}, optimal controlled processes of Problem \ref{prob:SBP} with fixed $\mathbb{Q}^{\rho}$ belong to the following class.
\begin{align*}
    \Pi(\Sigma_{ \mathrm{ini}},\Sigma_{ \mathrm{fin}};\mathcal{P}):=\left\{\mathbb{P}^{\pi} \in \Pi(\Sigma_{ \mathrm{ini}},\Sigma_{ \mathrm{fin}})|\pi \in \mathcal{P}\right\}
\end{align*}
This justifies the form \eqref{eq:controlled process driven by policy} of controlled processes in Proposition \ref{prop:relationship between MIODCP and GSB}.

\subsection{Optimal Reference Process for Fixed Controlled Process} \label{subsec:Optimal Reference Process for Fixed Bridge Process}

In Section \ref{subsec:Optimal Reference Process for Fixed Bridge Process}, we show that the optimal prior of Problem \ref{prob:MIODCP} with fixed $\pi \in \mathcal{P}$ yields the optimal reference process of Problem \ref{prob:SBP} with fixed controlled process distribution $\mathbb{P}^{\pi} \in \Pi(\Sigma_{ \mathrm{ini}},\Sigma_{ \mathrm{fin}};\mathcal{P})$.

\begin{prop}
    Consider a given policy $\pi \in \mathcal{P},\pi_{k}(\cdot|x)=\mathcal{N}(K_{k}x,V_{k})$, which satisfies $\mathbb{P}^{\pi} \in \Pi(\Sigma_{ \mathrm{ini}},\Sigma_{ \mathrm{fin}};\mathcal{P})$.
    Then, the unique optimal reference process distribution of Problem \ref{prob:SBP} with fixed controlled process distribution $\mathbb{P}^{\pi}$ is given by $\mathbb{Q}^{\rho^{\pi}}$.
    \label{prop:optimal reference process for fixed bridge process}
\end{prop}

\begin{proof}
    Because the second term of the objective function \eqref{eq:objective function of SBP} is independent of $\mathbb{Q}^{\rho}$, we only consider the first term $\mathcal{D}_{ \mathrm{KL}}\left[\mathbb{P}^{\pi}\|\mathbb{Q}^{\rho}\right]$.
    From the chain rule for the KL divergence with $\mathbb{P}^{\pi}_{0}=\mathbb{Q}_{0}^{\rho}=\mathcal{N}(0,\Sigma_{\mathrm{ini}})$,
    \begin{align*}
        &\mathcal{D}_{ \mathrm{KL}}\left[\mathbb{P}^{\pi}\|\mathbb{Q}^{\rho}\right]\\
        &= \sum_{k=0}^{T-1}\mathbb{E}\left[\mathcal{D}_{ \mathrm{KL}}\left[\mathcal{N}\left(B_{k}K_{k}x_{k},B_{k}V_{k}B_{k}^{\top}\right)\|\right.\right.\\
        &\left.\left.\hspace{40pt}\mathcal{N}\left(0,B_{k}R_{k}B_{k}^{\top}\right)\right]\right].
    \end{align*}
    It hence follows that
    \begin{align*}
        &\inf_{\mathbb{Q}^{\rho} \in \Pi(\mathcal{R})}\mathcal{D}_{ \mathrm{KL}}\left[\mathbb{P}^{\pi}\|\mathbb{Q}^{\rho}\right] \\
        &\Leftrightarrow\inf_{R_{k}\succeq 0}\mathbb{E}\left[\mathcal{D}_{ \mathrm{KL}}\left[\mathcal{N}\left(B_{k}K_{k}x_{k},B_{k}V_{k}B_{k}^{\top}\right)\|\right.\right. \\
        &\left.\left. \hspace{50pt}\mathcal{N}\left(0,B_{k}R_{k}B_{k}^{\top}\right)\right]\right], k\in \llbracket 0,T-1 \rrbracket.
    \end{align*}
    Noting that the above problem takes the same form as in \eqref{eq:MIODCP with fixed policy}, we can apply the same argument as in the proof of \cite[Proposition 2]{enami2025policy} and any optimal covariance matrix $R_{k}^{*}$ of the above problem satisfies
    \begin{align}
        B_{k}R_{k}^{*}B_{k}^{\top} = B_{k}(K_{k}X_{k}K_{k}^{\top} + V_{k})B_{k}^{\top}. \label{eq:condition of optimal reference process}
    \end{align}
    Because the covariance matrix $R_{k}^{\pi}$ of $\rho_{k}^{\pi}$ satisfies
    \begin{align}
        R_{k}^{\pi} = K_{k}X_{k}K_{k}^{\top} + V_{k} \label{eq:covariance matrix of optimal prior for fixed policy}
    \end{align}
    from Proposition \ref{prop:optimal prior of MIODCP with policy fixed}, $\mathbb{Q}^{\rho^{\pi}}$ is an optimal solution to Problem \ref{prob:SBP} with fixed $\mathbb{P}^{\pi}$.
    In addition, \eqref{eq:condition of optimal reference process} implies that any optimal covariance matrix $R_{k}^{*}$ yields the same transition covariance as $\mathbb{Q}^{\rho^{\pi}}$ for any $k\in\llbracket0,T-1\rrbracket$.
    Since the initial distribution of the reference process is fixed, any optimal solution $\{R_{k}^{*}\}_{k=0}^{T-1}$ therefore induces the same reference process distribution.
    Hence, $\mathbb{Q}^{\rho^{\pi}}$ is the unique optimal solution to Problem \ref{prob:SBP} with fixed $\mathbb{P}^{\pi}$.
    Therefore, the claim of Proposition \ref{prop:optimal reference process for fixed bridge process} holds.
\end{proof}

\subsection{Alternating Optimization Algorithm for Generalized Schr\"{o}dinger Bridges} \label{subsec:Alternating Optimization Algorithm for SBPs}

On the basis of Propositions \ref{prop:optimal bridge process for fixed reference process} and \ref{prop:optimal reference process for fixed bridge process}, we propose the following alternating optimization algorithm for Problem \ref{prob:SBP} as an analogue of Algorithm \ref{alg:alternating optimization algorithm for MIODCP}.

\begin{alg}
    Initialize prior $\rho^{(0)} \in \mathcal{R}_{+}$ arbitrarily, set the corresponding reference process distribution $\mathbb{Q}^{(0)}:=\mathbb{Q}^{\rho^{(0)}}$, and iterate the following steps alternately.
    \begin{description}
        \item[SB-step:]Calculate the policy $\pi^{(i)}:=\pi^{\rho^{(i)}}$ and set $\mathbb{P}^{{(i)}} := \mathbb{P}^{\pi^{(i)}}$. 
        \item[RR-step:]Calculate the prior $\rho^{(i+1)}:=\rho^{\pi^{(i)}}$ and set $\mathbb{Q}^{(i+1)} := \mathbb{Q}^{\rho^{(i+1)}}$.  
    \end{description}
    \label{alg:alternating optimization algorithm for SBPs}
\end{alg}

Note that if Algorithm \ref{alg:alternating optimization algorithm for MIODCP} is feasible under the assumptions that $B_{k}$ is full column rank and $W_{k}=0$ for any $k \in \llbracket0,T-1\rrbracket$, then  Algorithm \ref{alg:alternating optimization algorithm for SBPs} is also feasible.

\subsection{Convergence Analysis}
\label{subsec:Convergence Analysis 2}

In Section \ref{subsec:Convergence Analysis 2}, we will investigate whether a sequence $\{(\mathbb{P}^{(i)},\mathbb{Q}^{(i)})\}_{i\in\mathbb{Z}_{\geq 0}}$ generated by Algorithm \ref{alg:alternating optimization algorithm for SBPs} has a cluster point, whether such a cluster point is admissible, and what optimality properties it satisfies.
Using the connection between Problems \ref{prob:MIODCP} and \ref{prob:SBP} established earlier in Section \ref{sec:Alternating Optimization for Schrodinger Bridges}, these questions can be answered straightforwardly by applying Theorems \ref{thm:existence and admissibility of cluster point} and \ref{thm:coordinatewise optimality in MIODCP}.

\begin{thm}
     Suppose the assumptions of Theorem \ref{thm:existence and admissibility of cluster point}.
     In addition, suppose that $B_{k}$ is full column rank and $W_{k}=0$ for any $k \in \llbracket0,T-1\rrbracket$.
     Let $\{(\mathbb{P}^{(i)},\mathbb{Q}^{(i)})\}_{i\in\mathbb{Z}_{\geq 0}}$ be the sequence generated by Algorithm \ref{alg:alternating optimization algorithm for SBPs}.
     Then, the sequence has a cluster point, and any cluster point $(\mathbb{P}^{\circ},\mathbb{Q}^{\circ})$ is admissible for Problem \ref{prob:SBP}.
     In addition, for any cluster point $(\mathbb{P}^{\circ},\mathbb{Q}^{\circ})$,
    \begin{align}
        \mathbb{Q}^{\circ}\in\arg\min_{\mathbb{Q}^{\rho}\in\Pi(\mathcal{R})}
        \hat{J}(\mathbb{P}^{\circ},\mathbb{Q}^{\rho}). \label{eq:coordinate wise optimal reference process}
    \end{align}
    Furthermore, if the prior $\rho^{\circ}$ corresponding to $\mathbb{Q}^{\circ}$ belongs to $\mathcal{R}_{+}$, then
    \begin{align}
        \mathbb{P}^{\circ} \in \arg\min_{\mathbb{P}\in\Pi(\Sigma_{\mathrm{ini}},\Sigma_{\mathrm{fin}})}\hat{J}(\mathbb{P},\mathbb{Q}^{\circ}). \label{eq:coordinate wise optimal controlled process}
    \end{align}
\end{thm}

\begin{proof}
Let $\{(\pi^{(i)},\rho^{(i)})\}_{i\in\mathbb{Z}_{\geq0}}$ be the sequence generated by Algorithm \ref{alg:alternating optimization algorithm for MIODCP} from the same initial prior as Algorithm~\ref{alg:alternating optimization algorithm for SBPs}.
By construction, Algorithms \ref{alg:alternating optimization algorithm for MIODCP} and \ref{alg:alternating optimization algorithm for SBPs} generate corresponding sequences satisfying $\mathbb{P}^{{(i)}} := \mathbb{P}^{\pi^{(i)}}, \mathbb{Q}^{(i)} := \mathbb{Q}^{\rho^{(i)}}$ for any $i\in\mathbb{Z}_{\geq0}$.
By Theorem \ref{thm:existence and admissibility of cluster point}, the sequence $\{(\pi^{(i)},\rho^{(i)})\}_{i\in\mathbb{Z}_{\geq0}}$ has a cluster point.
Hence, there exists a subsequence $\{i_{j}\}$ and a cluster point $(\pi^{\circ},\rho^{\circ})$ such that
\begin{align*}
    (\pi^{(i_{j})},\rho^{(i_{j})}) \rightarrow (\pi^{\circ},\rho^{\circ}).
\end{align*}
Since the Gaussian process distributions $\mathbb{P}^{\pi}$ and $\mathbb{Q}^{\rho}$ induced by $\pi \in \mathcal{P}$ and $\rho \in \mathcal{R}$ depend continuously on the corresponding parameters of $\pi$ and $\rho$, respectively, we have
\begin{align*}
    (\mathbb{P}^{(i_{j})},\mathbb{Q}^{(i_{j})})\rightarrow (\mathbb{P}^{\pi^{\circ}}, \mathbb{Q}^{\rho^{\circ}}).
\end{align*}
Therefore, $\{(\mathbb{P}^{(i)},\mathbb{Q}^{(i)})\}_{i\in\mathbb{Z}_{\geq0}}$ has a cluster point.

We next consider any cluster point $(\mathbb{P}^{\circ},\mathbb{Q}^{\circ})$ of
$\{(\mathbb{P}^{(i)},$ $\mathbb{Q}^{(i)})\}_{i\in\mathbb{Z}_{\geq0}}$.
Let $\{i_{j}\}$ be a subsequence such that
\begin{align*}
    (\mathbb{P}^{(i_{j})},\mathbb{Q}^{(i_{j})}) \rightarrow (\mathbb{P}^{\circ},\mathbb{Q}^{\circ}).
\end{align*}
By repeating the compactness argument in the proof of Theorem \ref{thm:existence and admissibility of cluster point}, the corresponding subsequence $\{(\pi^{(i_{j})},$ $\rho^{(i_{j})})\}_{j\in\mathbb{Z}_{\geq0}}$ admits a further subsequence, denoted again by $\{i_{j}\}$, converging to a cluster point $(\pi^{\circ},\rho^{\circ})$ of Algorithm \ref{alg:alternating optimization algorithm for MIODCP}.
Then,
\begin{align*}
    \mathbb{P}^{(i_{j})} &= \mathbb{P}^{\pi^{(i_{j})}} \rightarrow \mathbb{P}^{\pi^{\circ}},\\
    \mathbb{Q}^{(i_{j})} &= \mathbb{Q}^{\rho^{(i_{j})}} \rightarrow \mathbb{Q}^{\rho^{\circ}}.
\end{align*}
Since the same subsequence converges to
$(\mathbb{P}^{\circ},\mathbb{Q}^{\circ})$, it follows that
\begin{align}
    \mathbb{P}^{\circ}= \mathbb{P}^{\pi^{\circ}},\mathbb{Q}^{\circ}=\mathbb{Q}^{\rho^{\circ}}.\label{eq:cluster point correspondence}
\end{align}
By Theorem \ref{thm:existence and admissibility of cluster point}, $(\pi^{\circ},\rho^{\circ})$ is admissible for Problem \ref{prob:MIODCP}.
Together with Proposition \ref{prop:relationship between MIODCP and GSB}, \eqref{eq:cluster point correspondence} implies that $(\mathbb{P}^{\circ},\mathbb{Q}^{\circ})$ is admissible for
Problem \ref{prob:SBP}.

We finally establish coordinate-wise optimality.
For fixed $\pi^{\circ}$, combining \eqref{eq:coordinate wise optimal prior} and Proposition \ref{prop:relationship between MIODCP and GSB} imply \eqref{eq:coordinate wise optimal reference process}.
On the other hand, for fixed $\rho^{\circ} \in \mathcal{R}_{+}$, combining \eqref{eq:coordinate wise optimal policy} and Lemma \ref{lem:Markov solution of SBP} and Proposition \ref{prop:relationship between MIODCP and GSB} imply \eqref{eq:coordinate wise optimal controlled process}, which completes the proof.
\end{proof}

\section{Numerical Experiment}\label{sec:Numerical Experiment}

Section \ref{sec:Numerical Experiment} shows a numerical example of Algorithms \ref{alg:alternating optimization algorithm for MIODCP} and \ref{alg:alternating optimization algorithm for SBPs}.
The terminal time is given by $T=5$.
The system is given by
\begin{align*}
    A_{k}=\begin{bmatrix}
        1 & 0.25\\ -0.1 & 0.9
    \end{bmatrix},
    B_{k}=\begin{bmatrix}
        1 & 0.2\\ 0.1 & 0.8
    \end{bmatrix}
    , W_{k}=0\ \forall k.
\end{align*}
The initial and terminal state covariances are given by
\begin{align*}
    \Sigma_{\mathrm{ini}}=\begin{bmatrix}
        1.5 & 0.3 \\ 0.3 & 0.4
    \end{bmatrix},
    \Sigma_{\mathrm{fin}}=\begin{bmatrix}
        0.35 & -0.18\\
        -0.18& 1.3
    \end{bmatrix}.
\end{align*}
The covariance matrix of the initial prior $\rho^{(0)}$ in Algorithm \ref{alg:alternating optimization algorithm for MIODCP} is given by
\begin{align*}
    R_{k}^{(0)}=I\ \forall k.
\end{align*}
Note that the above setting satisfies the assumptions of Theorem \ref{thm:existence and admissibility of cluster point}.
In addition, in this setting, $B_{k}$ is full column rank and $W_{k}=0$ for any $k \in \llbracket0,T-1 \rrbracket$, and thus Algorithms \ref{alg:alternating optimization algorithm for MIODCP} and \ref{alg:alternating optimization algorithm for SBPs} are equivalent.

In Fig. \ref{fig:state trajectory}, the sample paths of the state process under $\pi^{(i)}$ are shown for $i=0,200$.
These samples can also be interpreted as realizations from $\mathbb{P}^{(0)}$ and $\mathbb{P}^{(200)}$.
The colored lines indicate the samples, and the black ellipses are the $3\sigma$ covariance ellipses for $\mathcal{N}(0,\Sigma_{\mathrm{ini}}), \mathcal{N}(0,\Sigma_{\mathrm{fin}})$.
As can be seen, the state process under $\pi^{(200)}$ has less stochasticity than that under $\pi^{(0)}$.
In fact, the covariance matrices of $\rho^{(200)}$ are
\begin{align*}
    &R_{0}^{(200)} = \begin{bmatrix}
        0.008 & 0.005\\
        0.005 & 0.020
    \end{bmatrix},
    R_{1}^{(200)} = \begin{bmatrix}
        0.012 & 0.015\\
        0.015 & 0.044
    \end{bmatrix},\\
    &R_{2}^{(200)} = \begin{bmatrix}
        0.034 & 0.063\\
        0.063 & 0.148
    \end{bmatrix},
    R_{3}^{(200)} = \begin{bmatrix}
        0.075 & 0.139\\
        0.139 & 0.286
    \end{bmatrix},\\
    &R_{4}^{(200)} = \begin{bmatrix}
        0.128 & 0.220\\
        0.220 & 0.403
    \end{bmatrix},
\end{align*}
which are less than $R_{k}^{(0)}=I$.
Fig. \ref{fig:objective value} shows the trajectory of $J(\pi^{(i)},\rho^{(i)})$, which is consistent with monotonicity of Algorithm \ref{alg:alternating optimization algorithm for MIODCP}.
In addition, Fig. \ref{fig:rho convergence} shows the trajectories of $\|R_{k}^{(i+1)}-R_{k}^{(i)}\|$ for $k=0,\ldots,4$.
Since $\|R_{k}^{(i+1)}-R_{k}^{(i)}\|$ approaches zero for any $k$, Fig. \ref{fig:rho convergence} shows that the prior updates become asymptotically small.
This behavior is consistent with the asymptotic behavior established in the convergence analysis of Algorithm \ref{alg:alternating optimization algorithm for MIODCP}.

\begin{figure}[htbp]
    \begin{center}
    \begin{tabular}{c}   
      \begin{minipage}[t]{0.5\hsize}
      \centerline{\includegraphics[width=70mm]{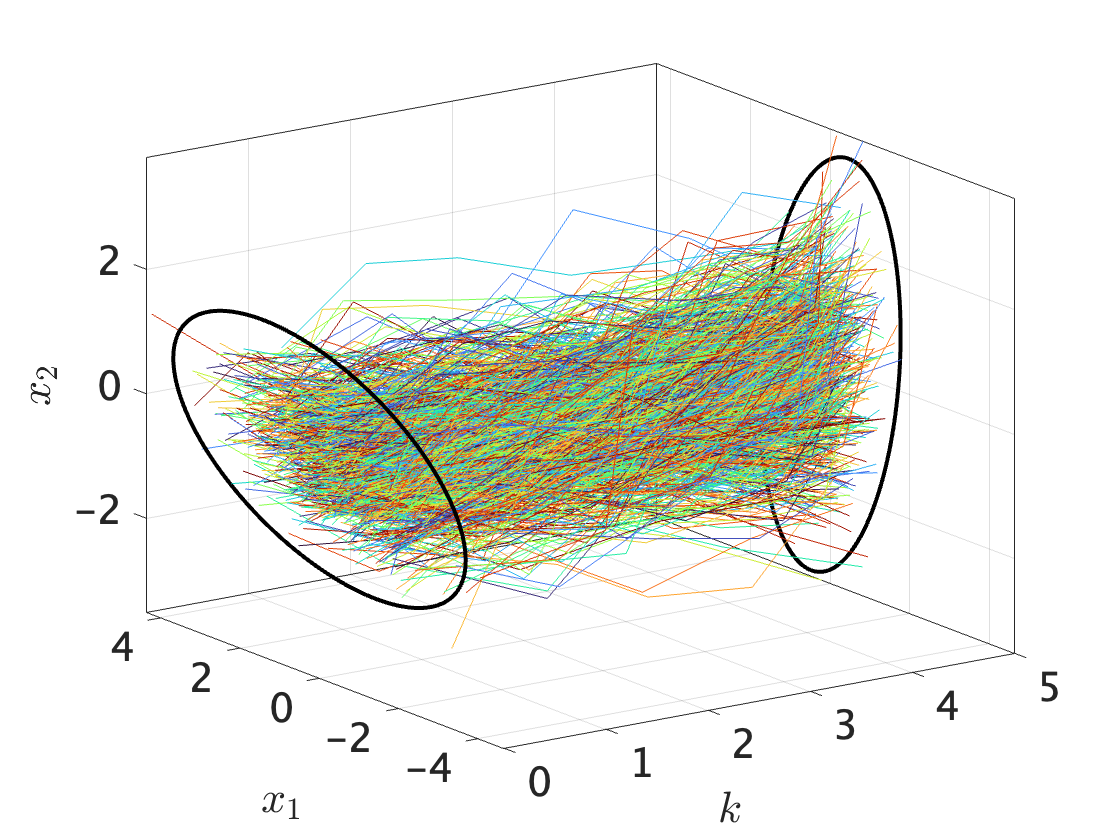}}
        \subcaption{$\pi^{(0)}$}
        \label{fig:state trajectory i_0}
      \end{minipage}\\
      
      \begin{minipage}[t]{0.5\hsize}
        \centerline{\includegraphics[width=70mm]{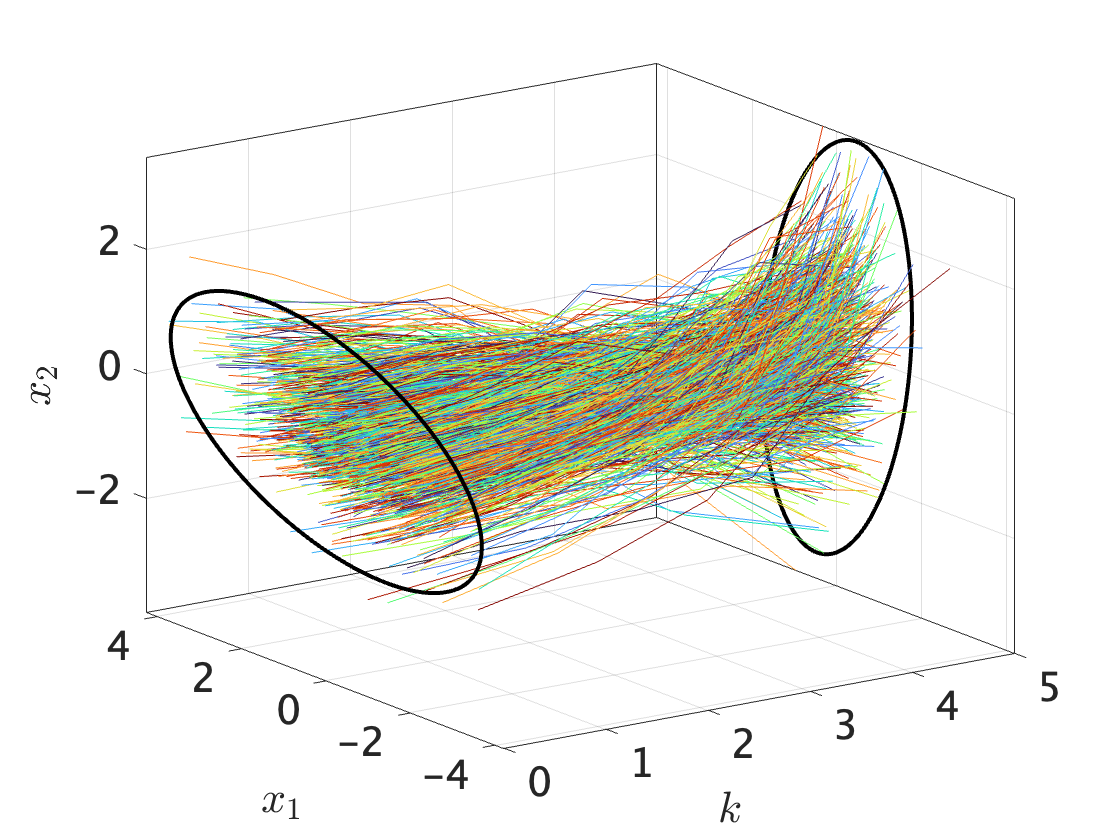}}
        \subcaption{$\pi^{(200)}$}
        \label{fig:state trajectory i_200}
      \end{minipage}
      
    \end{tabular}
    \caption{1000 samples of the state processes under $\pi^{(0)}$ and $\pi^{(200)}$.}
    \label{fig:state trajectory}
    \end{center}
\end{figure}

\begin{figure}[htbp]
    \begin{center}
    \centerline{\includegraphics[width=60mm]{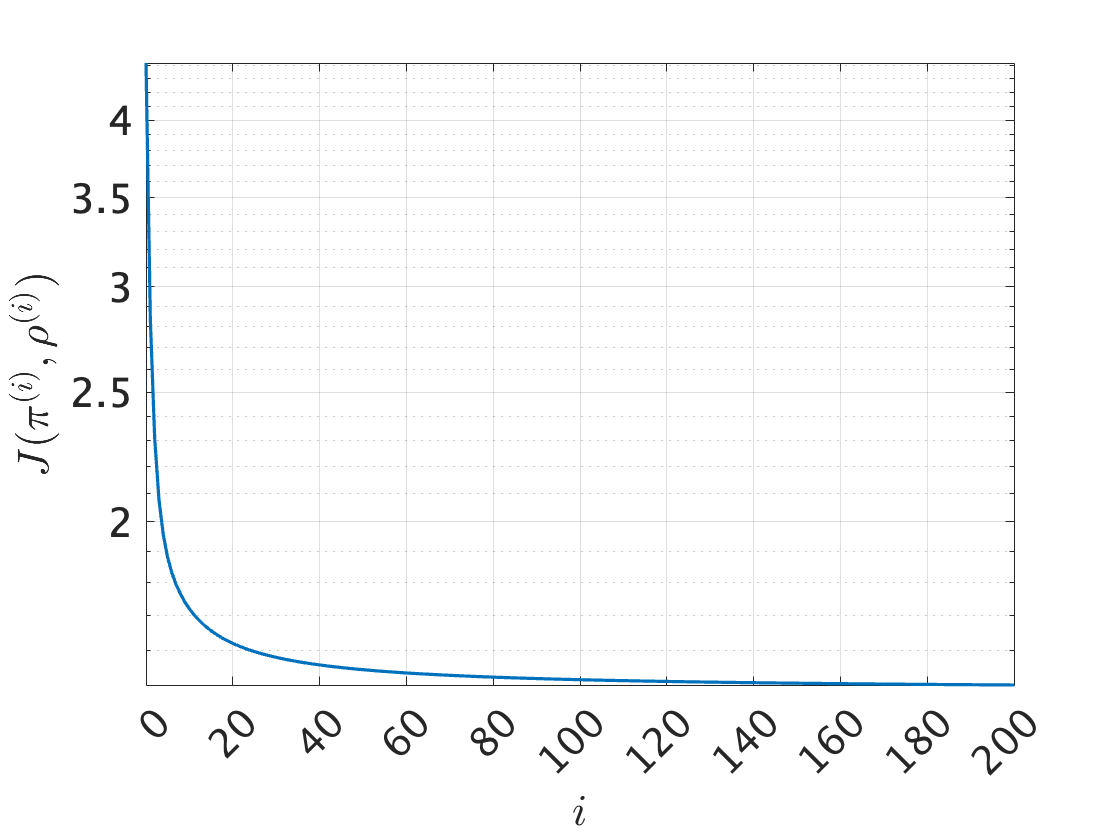}}
    \caption{The trajectory of the objective value $J(\pi^{(i)},\rho^{(i)})$.}
    \label{fig:objective value}
    \end{center}
\end{figure}

\begin{figure}[htbp]
    \begin{center}
    \centerline{\includegraphics[width=60mm]{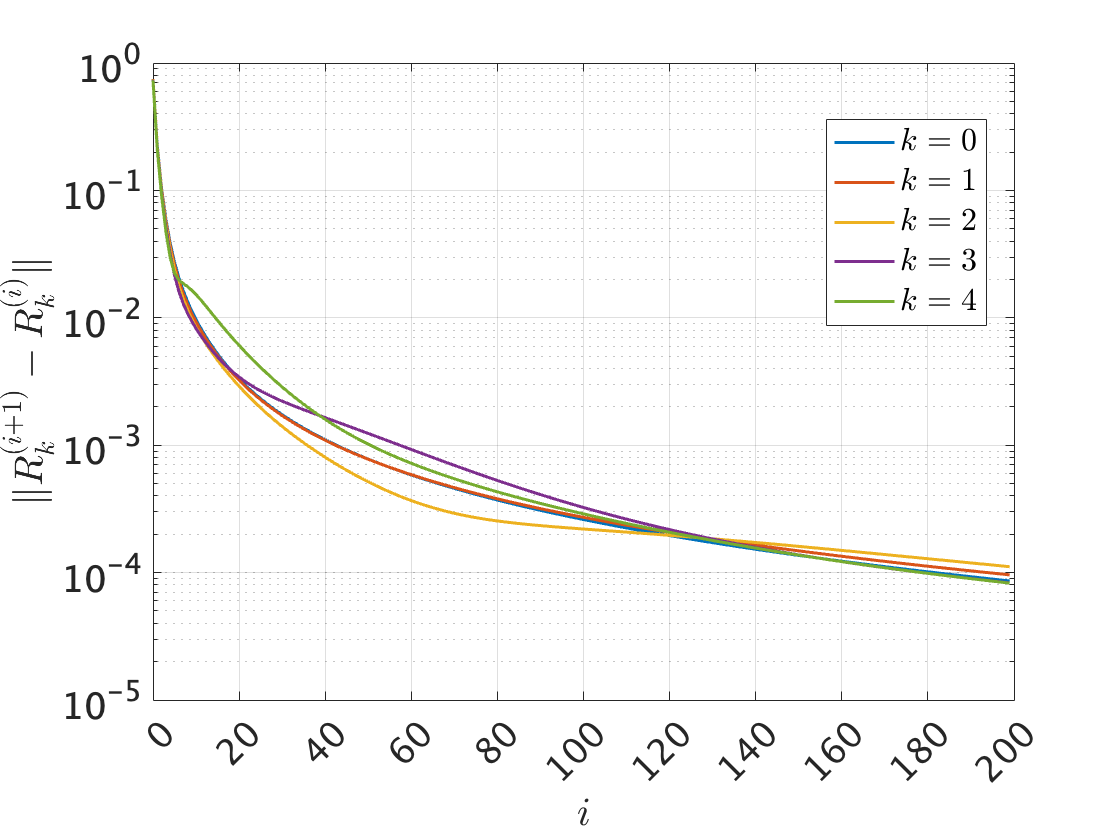}}
    \caption{The trajectories of the difference $\|R_{k}^{(i+1)}-R_{k}^{(i)}\|$ for $k=0,\ldots,4$.}
    \label{fig:rho convergence}
    \end{center}
\end{figure}

\section{Conclusion}\label{sec:Conclusion}
In this paper, we first investigated the MI optimal density control problem for discrete-time linear systems with Gaussian boundary conditions and Gaussian priors.
For this problem, we proposed an alternating optimization algorithm and analyzed convergence properties of it.
We also introduced the generalized Schr\"{o}dinger bridge problem with reference refinement.
Based on a relationship between these two problems, we constructed an alternating optimization algorithm for the Schr\"{o}dinger bridge problem and analyzed its convergence properties.
The numerical experiment illustrated the behavior of Algorithms \ref{alg:alternating optimization algorithm for MIODCP} and \ref{alg:alternating optimization algorithm for SBPs}.
Future work involves MI optimal density control in more general cases such as state-dependent cost and marginal constraints of Gaussian mixture models.

\begin{ack}                               
This work was supported by JSPS KAKENHI Grant Number 21H04875.  
\end{ack}

\appendix
\section*{Appendix}

\section{MI Optimal Density Control with Nonzero Initial and Terminal Means} \label{app:MI Optimal Density Control with Nonzero Initial and Terminal Means}

This appendix extends Problem \ref{prob:MIODCP} to the case where the initial and terminal distributions have nonzero means.
Consider
\begin{align}
    x_{0} \sim \mathcal{N} \left(\mu_{\mathrm{ini}},\Sigma_{\mathrm{ini}}\right), x_{T} \sim \mathcal{N}\left(\mu_{\mathrm{fin}},\Sigma_{\mathrm{fin}}\right),\label{eq:nonzero boundary means}
\end{align}
where $\mu_{\mathrm{ini}},\mu_{\mathrm{fin}}\in\mathbb{R}^{n}$.
Correspondingly, we extend the prior class $\mathcal{R}$ to
\begin{align}
    \tilde{\mathcal{R}}:= \{&\rho=\{\rho_{k}\}_{k=0}^{T-1} \mid\nonumber\\
    &\rho_{k}=\mathcal{N}\left(r_{k},R_{k}\right),r_{k}\in\mathbb{R}^{m},R_{k}\succeq0\}.
    \label{eq:app_prior_class}
\end{align}
We refer to Problem \ref{prob:MIODCP} with boundary condition \eqref{eq:nonzero boundary means} and prior class $\tilde{\mathcal{R}}$ as Problem \ref{prob:MIODCP}$^{\prime}$.

The following proposition shows that Problem \ref{prob:MIODCP}$^{\prime}$ can be decomposed into a deterministic optimal control problem for the state mean and Problem \ref{prob:MIODCP}.
It also implies that the mean of an optimal prior is the optimal input of the optimal control problem for the state mean.

\begin{prop}
    Suppose that Problem \ref{prob:MIODCP}$^{\prime}$ is admissible.
    Then, it can be decomposed into the optimal control problem of the state mean
    \begin{align}
        &\inf_{\{a_{k}\}_{k=0}^{T-1}} \sum_{k=0}^{T-1} \frac{1}{2} \|a_{k}\|^{2} \label{eq:objective of mean control} \\
        &\mathrm{s.t.}\  m_{k+1}=A_{k}m_{k}+B_{k}a_{k}, k\in \llbracket0,T-1 \rrbracket, \label{eq:system of mean control} \\
        &\hspace{17pt}m_{0} = \mu_{\mathrm{ini}}, m_{T} = \mu_{\mathrm{fin}}, \label{eq:boundary conditions of mean control}
    \end{align}
    and Problem \ref{prob:MIODCP}.
    In addition, if Problem \ref{prob:MIODCP}$^{\prime}$ has an optimal solution and $\{a_{k}^{*}\}_{k=0}^{T-1}$ denotes the unique optimal solution to \eqref{eq:objective of mean control}--\eqref{eq:boundary conditions of mean control}, then the mean $r_{k}^{*}$ of any optimal prior satisfies
    \begin{align}
        r_{k}^{*}=a_{k}^{*},k\in \llbracket0,T-1 \rrbracket .\label{eq:optimal prior mean}
    \end{align}
    \label{prop:nonzero mean MIODCP decomposition}
\end{prop}

\begin{proof}
    For an admissible policy $\pi$, recall the notations $m_{k}=\mathbb{E}[x_{k}], \check{x}_{k}:=x_{k}-m_{k}, a_{k}:=\mathbb{E}[u_{k}], \check{u}_{k}:=u_{k}-a_{k}$, which satisfy
    \begin{align}
        &m_{k+1}= A_{k}m_{k}+B_{k}a_{k}, \label{eq:mean system in proof}\\
        &m_{0}=\mu_{\mathrm{ini}},m_{T}=\mu_{\mathrm{fin}}, \label{eq:mean boundary conditions in proof}\\
        &\check{x}_{k+1}=\ A_{k}\check{x}_{k}+B_{k}\check{u}_{k}+w_{k}, \label{eq:deviation system in proof}\\
        &\check{x}_{0}\sim \mathcal{N}(0,\Sigma_{\mathrm{ini}}), \check{x}_{T}\sim \mathcal{N}(0,\Sigma_{\mathrm{fin}}) \label{eq:deviation boundary conditions in proof}
    \end{align}
    from \eqref{eq:linear system of MIODCP}.
    For any Borel set $\mathcal{A}\subset\mathbb{R}^{m}$, define the policy of the centered input by $\check{\pi}_{k}(\mathcal{A}|\check{x}):=\pi_{k}(\mathcal{A}+a_{k}|\check{x}+m_{k})$, where $\mathcal{A}+a_{k}:=\{u+a_{k}| u\in\mathcal{A}\}$.
    Similarly, define the translated prior as $\check{\rho}_{k}(\mathcal{A}):=\rho_{k}(\mathcal{A}+a_{k})$.
    By invariance of the KL divergence under a common translation,
    \begin{align}
        \mathcal{D}_{\mathrm{KL}} \left[ \pi_{k} ( \cdot | x_{k} )\| \rho_{k}\right] = \mathcal{D}_{\mathrm{KL}}\left[ \check{\pi}_{k}(\cdot|\check{x}_{k})\|\check{\rho}_{k} \right].  \label{eq:KL translation invariance}
    \end{align}
    Furthermore,
    \begin{align}
        \mathbb{E}\left[ \|u_{k}\|^{2}\right]= \|a_{k}\|^{2} +\mathbb{E} \left[\|\check{u}_{k}\|^{2}\right].  \label{eq:input cost decomposition}
    \end{align}
    Noting that we have $\check{\rho}_{k} = \mathcal{N}(s_{k},R_{k}),s_{k} := r_{k}-a_{k}$ for $\rho_{k}=\mathcal{N}(r_{k},R_{k})$, we next characterize the dependence of \eqref{eq:KL translation invariance} on the prior mean $s_{k}$.
    Since the original pair is admissible,
    \begin{align}
        \mathbb{E}\left[\mathcal{D}_{\mathrm{KL}}\left[\check{\pi}_{k}(\cdot|\check{x}_{k})\|\mathcal{N}( s_{k},R_{k}) \right]\right] <\infty. \label{eq:finite KL cost}
    \end{align}
    Thus, $\check{u}_{k}$ belongs to $s_{k}+\mathrm{Im}(R_{k})$ a.s.
    Taking its expectation and using $\mathbb{E}[\check{u}_{k}]=0$ gives $s_{k}\in \mathrm{Im}(R_{k})$.
    Consequently, $\mathcal{N}(s_{k},R_{k})$ and $\mathcal{N}(0,R_{k})$ are Gaussian distributions having the same affine support.
    On this support, their Radon--Nikodym derivative satisfies
    \begin{align}
        \log\frac{\mathrm{d}\mathcal{N}(0,R_{k})}{\mathrm{d}\mathcal{N}(s_{k},R_{k})}(\check{u}) =-\check{u}^{\top}R_{k}^{\dagger}s_{k}+\frac{1}{2}
        \|s_{k}\|_{R_{k}^{\dagger}}^{2}.\label{eq:RN derivative of Gaussians}
    \end{align}
    Therefore,
    \begin{align}
        &\mathbb{E}\left[\mathcal{D}_{\mathrm{KL}}\left[\check{\pi}_{k}( \cdot|\check{x}_{k})\|\mathcal{N}(s_{k}, R_{k})\right]\right]\nonumber\\
        &=\mathbb{E}\left[\mathcal{D}_{\mathrm{KL}}\left[ \check{\pi}_{k}(\cdot|\check{x}_{k})\|\mathcal{N}( 0,R_{k})\right]\right]\nonumber\\
        &\quad -\mathbb{E}[\check{u}_{k}^{\top}]R_{k}^{\dagger}s_{k}+\frac{1}{2}
        \|s_{k}\|_{R_{k}^{\dagger}}^{2}
        \nonumber\\
        &=\mathbb{E}\left[\mathcal{D}_{\mathrm{KL}}\left[ \check{\pi}_{k}(\cdot|\check{x}_{k})\|\mathcal{N}( 0,R_{k})\right]\right]+\frac{1}{2}
        \|s_{k}\|_{R_{k}^{\dagger}}^{2}. \label{eq:app_KL_mean_decomposition}
    \end{align}
Using
$s_{k}=r_{k}-a_{k}$,
the objective function can hence be written as
\begin{align}
    &J(\pi,\rho)\\
    &=\sum_{k=0}^{T-1}\frac{1}{2}\|a_{k}\|^{2}\nonumber\\
    &\quad+\mathbb{E}\left[ \sum_{k=0}^{T-1} \left\{\frac{1}{2} \|\check{u}_{k}\|^{2}+\mathcal{D}_{\mathrm{KL}}\left[ \check{\pi}_{k}(\cdot|\check{x}_{k})\| \mathcal{N}( 0,R_{k})\right]\right\}\right]\nonumber\\
    &\quad+\sum_{k=0}^{T-1}\frac{1}{2}\| r_{k}-a_{k}\|_{R_{k}^{\dagger}}^{2}.\label{eq:objective decomposition}
\end{align}
The first term depends only on the deterministic mean system \eqref{eq:mean system in proof} and the boundary condition \eqref{eq:mean boundary conditions in proof}.
The second term, together with \eqref{eq:deviation system in proof} and \eqref{eq:deviation boundary conditions in proof}, is exactly the objective and constraints of Problem \ref{prob:MIODCP} with zero mean prior $\mathcal{N}(0,R_{k})$.
The last term is nonnegative and equality holds if and only if $r_{k}=a_{k}$ since $r_{k}-a_{k}\in \mathrm{Im}(R_{k})$.
Hence, we can choose the mean of $\rho_{k}$ as $r_{k}=a_{k}$ without loss of generality.
Conversely, any feasible solution of the mean control problem \eqref{eq:objective of mean control}--\eqref{eq:boundary conditions of mean control} and any admissible pair of Problem \ref{prob:MIODCP} can be combined by translating the state and input distributions by $m_{k}$ and $a_{k}$, respectively, to construct an admissible pair of Problem \ref{prob:MIODCP}$^{\prime}$.
For this constructed pair, the last term vanishes, and its objective value is equal to the sum of the objective values of the mean control problem and Problem \ref{prob:MIODCP}.
Therefore, Problem \ref{prob:MIODCP}$^{\prime}$ is equivalently decomposed into the mean control problem
\eqref{eq:objective of mean control}--\eqref{eq:boundary conditions of mean control} and Problem \ref{prob:MIODCP}.

The mean control problem \eqref{eq:objective of mean control}--\eqref{eq:boundary conditions of mean control} has a strictly convex objective over an affine feasible set and therefore has a unique optimal input process if it is feasible.
Denote the optimal input by $\{a_{k}^{*}\}_{k=0}^{T-1}$.
At any optimal solution to Problem \ref{prob:MIODCP}$^{\prime}$, the last term in \eqref{eq:objective decomposition} must be zero and the mean input must minimize \eqref{eq:objective of mean control}.
Therefore, $r_{k}^{*}=a_{k}^{*}$ for any $k \in \llbracket0,T-1 \rrbracket$, which implies  \eqref{eq:optimal prior mean}.
\end{proof}

The mean control problem \eqref{eq:objective of mean control}--\eqref{eq:boundary conditions of mean control} is a standard linear-quadratic-regulator problem with a fixed terminal state and has been extensively studied.
For example, see \cite{lewis2012optimal}.
As referred to in Remark \ref{rem:admissibility of convex problem}, $G(T,0)\succ 0$ is a part of a sufficient condition to guarantee feasibility of the boundary constraint \eqref{eq:boundary conditions of MIODCP}.
If $G(T,0)\succ0$, the optimal feedforward input is explicitly
given by
\begin{align}
    a_{k}^{*} =B_{k}^{\top}\Phi(T,k+1)^{\top}G(T,0)^{-1}\left( \mu_{\mathrm{fin}}-\Phi(T,0)\mu_{\mathrm{ini}}\right), \nonumber
\end{align}
which also gives the closed form of $r_{k}^{*}$ by \eqref{eq:optimal prior mean}.

\section{Generalized Schr\"{o}dinger Bridges with Nonzero Marginal Means} \label{app:Generalized Schrodinger Bridges with Nonzero Margianl Means}

This appendix extends Problem \ref{prob:SBP} to the case where the initial and terminal marginal distributions have nonzero means.
Recall the prior class $\tilde{\mathcal{R}}$ introduced in Appendix \ref{app:MI Optimal Density Control with Nonzero Initial and Terminal Means}.
For $\rho\in\tilde{\mathcal{R}}$, consider the reference process
\begin{align}
    x_{k+1}=A_{k}x_{k}+B_{k}w_{k}^{\rho},w_{k}^{\rho}\sim\rho_{k},x_{0}\sim\mathcal{N}( \mu_{\mathrm{ini}},\Sigma_{\mathrm{ini}}),
    \label{eq:nonzero reference process}
\end{align}
where $w_{0}^{\rho},\ldots,w_{T-1}^{\rho}$ are mutually independent and independent of $x_{0}$.
Abuse the notation $\mathbb{Q}^{\rho}$ to denote the distribution of the above state process and define
\begin{align*}
    \Pi(\tilde{\mathcal{R}}):=\left\{\mathbb{Q}^{\rho}\mid \rho\in\tilde{\mathcal{R}}\right\}.
\end{align*}
In addition, denote by
\begin{align}
    \Pi\left((\mu_{\mathrm{ini}},\Sigma_{\mathrm{ini}}),(\mu_{\mathrm{fin}},\Sigma_{\mathrm{fin}})\right)
\end{align}
the set of all process distributions whose marginals at $k=0$ and $k=T$ are $\mathcal{N}(\mu_{\mathrm{ini}},\Sigma_{\mathrm{ini}})$ and $\mathcal{N}(\mu_{\mathrm{fin}},\Sigma_{\mathrm{fin}})$, respectively.

Now, we consider the following extension of Problem \ref{prob:SBP}.

\begin{prob}
    \label{prob:nonzero mean SBP}
    Find a pair of a controlled process distribution $\mathbb{P} \in\Pi\left((\mu_{\mathrm{ini}},\Sigma_{\mathrm{ini}}),(\mu_{\mathrm{fin}},\Sigma_{\mathrm{fin}})\right)$ and a reference process distribution $\mathbb{Q}^{\rho}\in\Pi(\tilde{\mathcal{R}})$ that solves
    \begin{align*}
        \inf_{\mathbb{P},\mathbb{Q}^{\rho}}\hat{J}(\mathbb{P},\mathbb{Q}^{\rho}).
    \end{align*}
\end{prob}

The following result shows that Problem \ref{prob:nonzero mean SBP} can be reduced to Problem \ref{prob:SBP} together with the state mean control problem \eqref{eq:objective of mean control}--\eqref{eq:boundary conditions of mean control}.

\begin{prop}
    \label{prop:nonzero mean SBP}
    Suppose that Problem \ref{prob:nonzero mean SBP} is admissible and $B_{k}$ has full column rank for any $k=0,\ldots,T-1$.
    Then, Problem \ref{prob:nonzero mean SBP} is equivalently decomposed into the state mean control problem \eqref{eq:objective of mean control}--\eqref{eq:boundary conditions of mean control} and Problem \ref{prob:SBP}.
    In addition, suppose that Problem \ref{prob:nonzero mean SBP} has an optimal solution $(\mathbb{P}^{*},\mathbb{Q}^{\rho^{*}})$.
    Let $\{a_{k}^{*}\}_{k=0}^{T-1}$ be the unique optimal solution to \eqref{eq:objective of mean control}--\eqref{eq:boundary conditions of mean control}.
    Then, the mean of any optimal $\rho_{k}^{*}=\mathcal{N}(r_{k}^{*},R_{k}^{*})$ satisfies
    \begin{align}
        r_{k}^{*}=a_{k}^{*},k\in \llbracket0,T-1 \rrbracket. \label{eq:optimal reference mean in SBP}
    \end{align}
\end{prop}

\begin{proof}
    Consider Problem \ref{prob:MIODCP}$^{\prime}$ introduced in Appendix \ref{app:MI Optimal Density Control with Nonzero Initial and Terminal Means} with $W_{k}=0$ for any $k=0,\ldots,T-1$.
    For any admissible pair $(\pi,\rho)$ of this problem, let $\mathbb{P}^{\pi}$ be the process distribution induced by
    \begin{align*}
        &x_{k+1}=A_{k}x_{k}+B_{k}u_{k},\\
        &u_{k}\sim\pi_{k}(\cdot|x_{k}),x_{0}\sim \mathcal{N}( \mu_{\mathrm{ini}},\Sigma_{\mathrm{ini}}).
    \end{align*}
    The proof of Proposition \ref{prop:relationship between MIODCP and GSB} does not rely on the means of the initial distribution or the prior.
    Therefore, the same argument gives
    \begin{align}
        \hat{J}(\mathbb{P}^{\pi},\mathbb{Q}^{\rho}) =J(\pi,\rho). \label{eq:objective correspondence in nonzero mean case}
    \end{align}
    Conversely, the same argument as in Lemma \ref{lem:Markov solution of SBP} shows that the controlled process can be restricted to Markov processes without loss of generality.
    For any admissible Markov controlled process distribution $\mathbb{P}$ of Problem \ref{prob:nonzero mean SBP} with fixed $\mathbb{Q}^{\rho}$, the converse part of Proposition \ref{prop:relationship between MIODCP and GSB} therefore yields an admissible policy $\pi$ of Problem \ref{prob:MIODCP}$^{\prime}$ such that
    \begin{align*}
        \mathbb{P} =\mathbb{P}^{\pi},\hat{J}(\mathbb{P},\mathbb{Q}^{\rho})=J(\pi,\rho).
    \end{align*}
    Hence, Problem \ref{prob:nonzero mean SBP} and Problem \ref{prob:MIODCP}$^{\prime}$ have the same optimal value, and their admissible solutions correspond in the above sense.
    By Proposition \ref{prop:nonzero mean MIODCP decomposition}, Problem \ref{prob:MIODCP}$^{\prime}$ is equivalently decomposed into the mean control problem \eqref{eq:objective of mean control}--\eqref{eq:boundary conditions of mean control} and Problem \ref{prob:MIODCP}.
    Furthermore, the same proposition shows that, for any optimal solution, $r_{k}^{*}=a_{k}^{*}$ for any $k \in \llbracket0,T-1 \rrbracket$.
    Finally, by Proposition \ref{prop:relationship between MIODCP and GSB}, Problem \ref{prob:MIODCP} with $W_{k}=0$ and full column rank $B_{k}$ corresponds to Problem \ref{prob:SBP}.
    Consequently, Problem \ref{prob:nonzero mean SBP} is equivalently decomposed into the mean control problem \eqref{eq:objective of mean control}--\eqref{eq:boundary conditions of mean control} and Problem \ref{prob:SBP}, and \eqref{eq:optimal reference mean in SBP} holds.
\end{proof}

\bibliographystyle{plain}        
\bibliography{Automatica_MIODCP}  

\end{document}